\documentclass[reqno,11pt]{amsart}

\usepackage{mathdots}
\usepackage{tikz}
\usepackage{tikz-cd}
\usepackage{amssymb}
\usepackage{amsgen}
\usepackage{amsmath}
\usepackage{amsthm}
\usepackage{mathrsfs}
\usepackage{cite}
\usepackage{amsfonts}

\newcommand{\dom}{\mathop{\boldsymbol d}}
\newcommand{\ran}{\mathop{\boldsymbol r}}

\newcommand{\inv}{^{-1}}

\newcommand{\ov}[1]{\ensuremath{\overline {#1}}}
\newcommand{\til}[1]{\ensuremath{\widetilde {#1}}}

\newcommand{\wh}{\widehat}

\newcommand{\Hom}{\mathop{\mathrm{Hom}}\nolimits}

\usepackage{xcolor}

\newtheorem{Thm}{Theorem}[section]
\newtheorem{Prop}[Thm]{Proposition}

\newtheorem{Lemma}[Thm]{Lemma}
{\theoremstyle{definition}
}
{\theoremstyle{remark}
\newtheorem{Rmk}[Thm]{Remark}}
\newtheorem{Cor}[Thm]{Corollary}
{\theoremstyle{remark}
}
{\theoremstyle{remark}
}

{\theoremstyle{remark}
}
{\theoremstyle{remark}
}
{\theoremstyle{remark}
}
{\theoremstyle{remark}
\newtheorem*{Claim*}{Claim}}

\numberwithin{equation}{section}

\title{Factoring the Dedekind-Frobenius determinant of a semigroup}

\author{Benjamin Steinberg}
\address[B.~Steinberg]{%
    Department of Mathematics\\
    City College of New York\\
    Convent Avenue at 138th Street\\
    New York, New York 10031\\
    USA}
\email{bsteinberg@ccny.cuny.edu}

\thanks{The author was supported by a PSC CUNY grant.}
\date{\today}


\begin{document}

\begin{abstract}
The representation theory of finite groups began with Frobenius's factorization of Dedekind's group determinant.  In this paper, we consider the case of the semigroup determinant.  The semigroup determinant is nonzero if and only if the complex semigroup algebra is Frobenius, and so our results include applications to the study of Frobenius semigroup algebras.   We explicitly factor the semigroup determinant for  commutative semigroups and inverse semigroups.  We recover the Wilf-Lindstr\"om factorization of the semigroup determinant of a meet semilattice and Wood's factorization for a finite commutative chain ring.  The former was motivated by combinatorics and the latter by coding theory over finite rings.  We prove that the algebra of the multiplicative semigroup of a finite Frobenius ring is Frobenius over any field whose characteristic doesn't divide that of the ring.  As a consequence we obtain an easier proof of Kov\'acs's theorem that the algebra of the monoid of matrices over a finite field is a direct product of matrix algebras over group algebras of general linear groups (outside of the characteristic of the finite field).
\end{abstract}

\maketitle

\section{Introduction}
Dedekind in the 1880s introduced the group determinant of a finite group $G$:  you build a $G\times G$ matrix whose $(g,h)$-entry is $x_{gh}$ (where the $x_k$, $k\in G$, are variables) and you compute the determinant.  This is a homogeneous polynomial of degree $|G|$ and Dedekind was initially interested in this determinant as a means to compute the discriminant of a finite Galois extension of the field of rational numbers.  Dedekind was able to factor the group determinant of an abelian group into distinct linear factors but was unable to factor the determinant in the case of a nonabelian group.  He proposed the problem to Frobenius, who famously invented the character theory of nonabelian groups in order to solve it.  See~\cite{ConradRep} for a wonderful introduction to this subject and its history.

At roughly the same time, Smith~\cite{smith} computed the determinant of quite a different multiplication table: the multiplication table of $\{1,\ldots, n\}$ under the associative and commutative binary operation of $\gcd$ (greatest common divisor).  Unlike the case of Dedekind, he literally viewed the multiplication table as a matrix of numbers (rather than of variables) and took the determinant.  Of course, it is unlikely that he viewed this matrix as a multiplication table at the time.  Smith computed this determinant as $\phi(1)\phi(2)\cdots \phi(n)$ where $\phi$ is Euler's totient function.  The set $\{1,\ldots,n\}$ does not have an identity element with respect to $\gcd$ for $n\geq 2$, and so this is a \textit{bona fide} semigroup determinant.  In fact, one can immediately generalize Dedekind's group determinant to semigroups and Smith's determinant is obtained by specializing the variable $x_i$ to the number $i$.

The greatest common divisor as a binary operation on $\{1,\ldots, n\}$ is a special case of taking a meet semilattice (partially ordered set with binary greatest lower bounds or meets) and making it a semigroup via the meet operation.  The semigroup determinant of a meet semilattice was computed independently by Wilf~\cite{Wilf} and Lindstr\"om~\cite{latticedet2}.  In fact, Wilf assumed he was working with a lattice, but it doesn't matter for the computation.  Both authors showed that the semigroup determinant in this case factors into distinct linear polynomials that can be computed from the M\"obius function of the semilattice.  In the case of the $\gcd$-semilattice, this is essentially the number theoretic M\"obius function and a simple M\"obius inversion argument converts the Wilf-Lindstr\"om determinant into the Smith determinant.  Wilf was motivated in~\cite{Wilf} by computing determinants of certain combinatorially defined matrices which could be viewed as specializations of semigroup determinants of lattices.

Wood factored the semigroup determinant of a finite commutative chain ring in~\cite{woodsemigroup}.  These are finite commutative rings in which the ideals form a chain or, equivalently, are commutative local rings with a principal maximal ideal.  Wood computed the determinant here of the multiplicative semigroup and showed that it factored into linear factors, but there are multiplicities.  His motivation was a program to generalize the MacWilliams extension theorem for codes over a finite field to chain rings.  He told me (private communication) that one of the key issues was showing that the semigroup determinant does not vanish in this case.

It should be mentioned that Johnson has considered determinants of Latin squares~\cite{LatinSquare} (or quasigroups), but here we are only interested in the associative case.

The first thing we investigate in this paper is the question of when the semigroup determinant vanishes.  Unlike the group determinant, it is often the case that a semigroup determinant does vanish and the key observation is that the semigroup determinant of $S$ is nonzero if and only if $\mathbb CS$ is a Frobenius algebra (and, in particular, unital).  This is appears in a slightly different language in~\cite[Chapter~16, Proposition~18]{oknisemigroupalgebra}. 

It is often the case that a semigroup algebra is isomorphic to the algebra of a nicer structure, like a structured finite category, or to a direct product of semigroup algebras of simpler semigroups~\cite{mobius2,Stein2}.  This happens precisely because idempotents in the semigroup can be complemented in the semigroup algebra.  So our main approach is to work with Frobenius's paratrophic determinant~\cite{Frobeniushimself}, which is the natural generalization of the semigroup determinant to finite dimensional algebras with a distinguished basis, and then understand how this determinant is affected by a change of basis and how it behaves under direct product.  This allows us to exploit our isomorphisms of semigroup algebras with seemingly better-behaved algebras to compute the determinant.

Our most significant contribution is a factorization of the semigroup determinant of a commutative semigroup.  We prove that either the semigroup determinant is identically zero, or it factors into linear polynomials and we describe the factors and their multiplicities explicitly.  A number of the ideas used in our factorization result go back to work of Ponizovski\u{\i}~\cite{PoniFrob} and Wenger~\cite{Wenger} on commutative semigroups with Frobenius semigroup algebras.  But their results are not explicit enough to write down the exact factorization of the semigroup determinant, and so we need to refine and strengthen their results.  In particular, we explicitly realize the decomposition of the semigroup algebra  of a commutative semigroup (when it is unital) into a product of local rings by identifying the local rings as twisted contracted monoid algebras of commutative nilpotent semigroups with adjoined identities.  We give an example of a  commutative semigroup so that one of the twisted algebras is non-Frobenius even though the untwisted version of the algebra is Frobenius, and so the twist plays an important role.  Wenger was only able to deal with the case that all these algebras were untwisted in~\cite{Wenger}.

In coding theory, it is now known that finite Frobenius rings (not necessarily commutative) are exactly the ones for which the MacWilliams extension theorem holds for the Hamming weight.  We show here that if $R$ is a finite Frobenius ring and $K$ is a field whose characteristic does not divide that of $R$, then $KR$ is a Frobenius algebra (where we take the semigroup algebra with respect to the multiplicative monoid of $R$).  In particular, the semigroup determinant of $R$ does not vanish.  Since the ring of matrices over a finite field is a finite Frobenius ring, this allows us to obtain a new, and much easier, proof of Kov\'acs's theorem~\cite{Kovacs} that the algebra $KM_n(F)$ is isomorphic to a direct product of matrix algebras over group algebras of general linear groups if $F$ is a finite field of characteristic different from that of $K$ (this result was first proved in characteristic zero in~\cite{putchasemisimple}).

We also prove that the semigroup determinant of an inverse semigroup can be computed as the semigroup determinant of a finite groupoid, the latter of which is amenable to the same techniques as for groups.  Inverse semigroups are a class of semigroups generalizing both groups and meet semilattices.  They are precisely semigroups of partial permutations of a set closed under inversion.  The classical example is the rook monoid~\cite{Solomonrook} of all $0/1$-matrices with no two ones in the same row or column.

The paper is roughly organized as follows.  We being with a treatment of Frobenius's paratrophic determinant and how it relates to the modern definition of Frobenius algebra (going back to Brauer, Nesbitt and Nakayama); in particular, we observe that the semigroup determinant vanishes precisely when the semigroup algebra is not Frobenius, allowing us to give some simple necessary conditions for the semigroup determinant not to vanish.  Included are some basic computational results that were probably known to Frobenius.  The main result here is the change of basis formula, which then allows us to deduce quickly Frobenius's classical theorem for group determinants and the Wilf-Lindstr\"om result for semilattices, the latter from Solomon's description of the semigroup algebra of a semilattice~\cite{Burnsidealgebra}.  The following section shows how to use the author's isomorphism theorem for inverse semigroup algebras~\cite{mobius2} to compute the semigroup determinant of an inverse semigroup as the determinant of its associated groupoid after a change of variables involving the M\"obius function of the semigroup.  Section~\ref{s:Frob.ring} shows that the semigroup algebra of the multiplicative semigroup of a finite Frobenius ring is a Frobenius algebra outside of characteristics dividing the characteristic of the finite ring. This section also includes our new proof of Kov\'acs's theorem on the algebra $KM_n(F)$.  The following section is devoted to the determinant of a monoid obtained by adjoining an identity to a nilpotent semigroup.  It is widely believed that an overwhelming majority of monoids of order $n$ are obtained in this way.  For monoids obtained by adjoining an identity to a nilpotent semigroup, we compute the semigroup determinant explicitly, refining a result of Wenger~\cite{Wenger}.  The problem of whether such a monoid has a Frobenius algebra is equivalent to the problem of determining whether a $0/1$-matrix is singular.  The final section deals with the case of commutative semigroups.   We first show that if $S$ has central idempotents and $S^2=S$, then $KS$ is unital for any field and give a direct product decomposition into algebras of monoids whose non-invertible elements are nilpotent.  This applies, in particular, to the commutative case, where the decomposition was first shown by Ponizovski\u{\i}~\cite{PoniFrob}, but without a concrete isomorphism such as we provide.  For a commutative monoid in which all non-invertible elements are nilpotent, we explicitly compute the direct product decomposition of the monoid algebra into a direct product of local rings and we show these rings are twisted monoid algebras of nilpotent semigroups with adjoined identities.  We then use the results of the previous section (which work also in the twisted setting) to compute the determinant of a commutative semigroup.  As a corollary we deduce Wood's factorization for finite chain rings.

\section{The paratrophic determinant of a based algebra}
In this paper all algebras are associative and over $\mathbb C$ unless otherwise stated.  We do not assume the existence of identities in algebras.  Much of what we say can be made to work over arbitrary fields with the proviso that group algebras become more complicated in the modular setting and that nonzero polynomials can vanish identically over finite fields, but we stick to the complex numbers for simplicity and because of our initial motivations.  The reader is referred to~\cite{serrerep} for the character theory of finite  groups and to~\cite{benson} for finite dimensional algebras.

\subsection{General notions}
By a \emph{based algebra} we mean a finite dimensional $\mathbb C$-algebra $A$ with a distinguished basis $B$.   We often write $(A,B)$ for the pair.   The multiplication on $A$ is determined by its \emph{structure constants} with respect to $B$, defined by the equations
\[bb' = \sum_{b''\in B}c_{b'',b,b'}b''\] where $b,b'\in B$ and $c_{b'',b,b'}\in \mathbb C$.  Let $X_B=\{x_b\mid b\in B\}$ be a set of variables in bijection with $B$.  By the \emph{Cayley table} of $(A,B)$, we mean the $B\times B$ matrix over $\mathbb C[X_B]$ with
\[C(A,B)_{b,b'} = \sum_{b''\in B}c_{b'',b,b'}x_{b''}.\] Matrices obtained by specializing $C(A,B)$ at elements of $\mathbb C^B$ are called \emph{paratrophic matrices} (by Frobenius~\cite{Frobeniushimself}), and so $C(A,B)$ is the generic paratrophic matrix.  We define
\[\theta_{(A,B)}(X_B) = \det C(A,B).\]  It is either identically zero, or a homogeneous polynomial of degree $|B|$.  Frobenius~\cite{Frobeniushimself} called this the \emph{paratrophic determinant} of $(A,B)$.  Some of the results in this section about based algebras are likely in the work of Frobenius in a different language, and so we have included proofs for the convenience of the reader.

If $S$ is a semigroup, $A=\mathbb CS$ and $B=S$, then we put $\theta_S=\theta_{(\mathbb CS,S)}$ and call it the \emph{(Dedekind-Frobenius) semigroup determinant} of $S$. If we put $C(S)=C(\mathbb CS,S)$, then $C(S)_{s,t} = x_{st}$, and so $C(S)$ is an encoding of the multiplication table of $S$. Ok\'ninski calls it the Nakayama matrix of parameters of $S$ in~\cite[Chapter~16]{oknisemigroupalgebra}. When $G$ is a group, the group determinant $\theta_G$ was introduced by Dedekind in correspondence with Frobenius, the latter of whom first factored it into irreducibles, whence the name~\cite{ConradRep}.  If the semigroup $S$ is fixed, we often write $X$ instead of $X_S$.

Wilf factored the Dedekind-Frobenius determinant of a finite lattice, made into a semigroup via its meet operation~\cite{Wilf}; his interest was to give a method to compute the determinant of certain combinatorially defined matrices. This factorization was independently discovered by Lindstr\"{o}m~\cite{latticedet2} in the slightly more general case of meet semilattices.  An important special case is the so-called Smith determinant~\cite{smith}, which is the semigroup determinant of the semilattice $\{1,\dots, n\}$ with binary operation the greatest common divisor specialized at $x_i=i$, i.e., the determinant of the $n\times n$ matrix with $ij$-entry $\gcd(i,j)$.  Wood~\cite{woodsemigroup} factored the semigroup determinant of the multiplicative monoid of a finite commutative chain ring in connection with an approach to generalizing the MacWilliams' extension theorem from codes over finite fields to codes over finite rings.

 The first issue to deal with is when $\theta_{(A,B)}$ vanishes identically. This will show in one fell swoop that many semigroup determinants vanish. Recall that a finite dimensional unital $K$-algebra $A$ is \emph{Frobenius} if there is a linear mapping $\lambda\colon A\to K$ (called a \emph{Frobenius form}) such that the bilinear form $(a,b)\mapsto \lambda(ab)$ is nondegenerate or, equivalently, $\ker \lambda$ contains no nonzero left or right ideal of $A$~\cite{benson}.  Tensor products and direct products of Frobenius algebras are Frobenius and matrix algebras over a field are Frobenius (you can use the trace as $\lambda$).  Hence all semisimple complex algebras are Frobenius. Note that $A^*=\Hom_K(A,K)$ is a left $A$-module via $(af)(b) = f(ba)$.  Another characterization of Frobenius algebras is as those unital algebras for which $A^*$ is isomorphic to $A$ as a left $A$-module~\cite{Lam2}.

 The following theorem (in the unital setting) is what Lam calls Frobenius's criterion~\cite[16.83]{Lam2}, although it precedes the modern definition, which arose in the work of Brauer, Nesbitt and Nakayama.

\begin{Thm}\label{t:vanish.frob}
 Let $(A,B)$ be a based algebra.  Then $\theta_{(A,B)}\neq 0$ if and only if $A$ is a Frobenius algebra (and, in particular, unital).
\end{Thm}
\begin{proof}
Let $B=\{b_1,\ldots, b_n\}$ and suppose that the structure constants are given by $b_ib_j=\sum_{k=1}^n c_{kij} b_k$.  Let $\lambda\colon A\to \mathbb C$ be a linear map.  Then $\lambda(b_ib_j) =\sum_{k=1}^n c_{kij}\lambda(b_k)$, and so we see that specializing $x_b$ to $\lambda(b)$ turns $C(A,B)$ into the matrix of the bilinear from $(a,a')\mapsto \lambda(aa')$ with respect to the basis $B$.  Therefore, there is a $\lambda$ giving a nondegenerate form if and only if $\theta_{(A,B)}$ does not vanish identically.  Since $\mathbb C$ is an infinite field, this is equivalent to $\theta_{(A,B)}\neq 0$.

It follows immediately that if $A$ is Frobenius, then $\theta_{(A,B)}\neq 0$. For the converse, all that remains to show is that if $\theta_{(A,B)}\neq 0$, then $A$ is unital.  By the previous paragraph, we can find $\lambda\colon B\to \mathbb C$ such that $(a,a')\mapsto \lambda(aa')$ is a nondegenerate bilinear form with matrix $C$ (with respect to $B$) obtained by specializing $C(A,B)$ at the variables $x_b=\lambda(b)$.  Hence we can find a $\eta\in \mathbb C^n$ such that $(C\eta)_k=\lambda(b_k)$, for $k=1,\ldots, n$, by nonsingularity of $C$.  Let $e=\sum_{k=1}^n\eta_kb_k\in A$. Note that $\lambda(b_ke) = (C\eta)_k=\lambda(b_k)$ by choice of $e$.  Therefore, we have, for any $i,j$,
\[\lambda(b_ib_je) =\sum_{i=1}^k c_{kij}\lambda(b_ke)=\sum_{i=1}^k c_{kij}\lambda(b_k) = C_{ij}= \lambda(b_ib_j).\]  Fixing $j$, this  yields $\lambda(b_i(b_j-b_je))=0$ for all $i$ and hence $b_j=b_je$ by nondegeneracy of $(a,a')\mapsto \lambda(aa')$.  Since $j$ was arbitrary, we deduce that $e$ is a right identity for $A$.  A symmetric argument shows that $A$ has a left identity and hence $A$ is unital.
\end{proof}

The proof that $A$ is unital is modelled after~\cite[Chapter~16, Proposition~22]{oknisemigroupalgebra}, which considers the special case of contracted semigroup algebras.  Notice that the above proof shows that if $A$ is a finite dimensional $K$-algebra and there is a functional $\lambda\colon A\to K$ whose kernel contains no  left or right ideal, then $A$ must be unital.

\begin{Cor}\label{c:squared.algebra}
Let $(A,B)$ be a based algebra.  If $\theta_{(A,B)}\neq 0$, then $A^2=A$.  In particular, if $S$ is a semigroup and $S^2\neq S$, then $\theta_S=0$.
\end{Cor}
\begin{proof}
The first statement is obvious from Theorem~\ref{t:vanish.frob} since $\theta_{(A,B)}\neq 0$ implies $A$ is unital.  The second statement follows because $(\mathbb CS)^2=\mathbb CS^2$.
\end{proof}

It follows from Theorem~\ref{t:vanish.frob} that the paratrophic determinant of a semisimple algebra never vanishes.  We shall investigate this case in more detail shortly.  Note that Brauer showed that if $A$ is unital, then  the intertwiners between $A$ and $A^*$ are exactly the paratrophic matrices, i.e., those matrices obtained by specializing $C(A,B)$ at elements of $\mathbb C^B$.

Unlike the case of the group determinant, which never vanishes, the semigroup determinant can be identically zero and figuring out when this happens is at least as hard as determining if a $0/1$-square matrix is nonsingular. For starters, we show that, for von Neumann regular semigroups, the semigroup determinant often vanishes.  A semigroup $S$ is (von Neumann) \emph{regular} if, for all $s\in S$, there exists $t\in S$ with $sts=s$.  For example, the full transformation monoid $T_n$ of all maps on an $n$-element set is regular, as is the multiplicative monoid of $M_n(F)$ for a finite field $F$.  It follows from Theorem~\ref{t:vanish.frob} and~\cite[Chapter~16, Theorem~17]{oknisemigroupalgebra} (or, for monoids,~\cite[Theorem~15.6]{repbook}) that if $S$ is a finite regular semigroup, then $\theta_S\neq 0$ if and only if $\mathbb CS$ is Frobenius, if and only if $\mathbb CS$ is semisimple.  But many classes of regular semigroups do not have semisimple algebras; for example $\mathbb CT_n$ is not semisimple for $n\geq 2$~\cite{repbook}.  Recall that a semigroup is a \emph{band} if each element is idempotent.  A commutative band is called a \emph{semilattice} since it can be endowed with a partial order such that the product is the meet.  The representation theory of bands has been extensively studied in connection with hyperplane arrangements and Markov chains~\cite{BHR,Brown1,Brown2,tha,MSS,ourmemoirs}. A band is, of course, regular and has a semisimple algebra if and only if it is commutative; see~\cite[Chapter~5]{CP} (or~\cite{repbook} for the case of monoids).

\begin{Cor}
The semigroup determinant of $T_n$ vanishes for all $n\geq 2$.  If $S$ is a band, then the semigroup determinant of $S$ is nonzero if and only if $S$ is a semilattice.
\end{Cor}

An easy to check necessary condition for the semigroup determinant  not to vanish is the following.

\begin{Cor}\label{c:intertwiner}
Let $S$ be a semigroup.  A necessary condition for $\theta_S\neq 0$ (i.e., for $\mathbb CS$ to be Frobenius) is that each element $s\in S$ fixes the same number of elements of $S$ under both left multiplication and right multiplication. 
\end{Cor}
\begin{proof}
If $\mathbb CS$ is Frobenius, then $\mathbb CS^*\cong \mathbb CS$ as left $\mathbb CS$-modules, and  therefore they have the same character.  The character of $\mathbb CS$ as a left module counts the number of fixed points of $s$ under left multiplication.  The character of $\mathbb CS$ as a right module counts the number of fixed points of $s$ under right multiplication.  Since taking the transpose of a matrix doesn't change the trace, the character of $\mathbb CS^*$ also counts the number of fixed points of $s$ under right multiplication.  The result follows.
\end{proof}

For example, in $T_n$, a constant mapping has one fixed point (itself) when multiplying on the left and $n$ fixed points (the constant mappings) when acting on the right, and so $\mathbb CT_n$ is not Frobenius by Corollary~\ref{c:intertwiner} for $n\geq 2$.  On the other hand, if $M$ is a commutative monoid every element has the same number of fixed points on the left and right.  But not all commutative monoids have a Frobenius algebra.  For example, if $M$ is the commutative monoid with generators $x,y$ and relations saying that all  products of length $2$ or more are equal, then $\mathbb CM\cong \mathbb C\times \mathbb C[x,y]/(x,y)^2$ is not Frobenius; see also Section~\ref{s:nilpotents}.  So the condition in Corollary~\ref{c:intertwiner} is necessary but not sufficient.

It turns out that in order to compute certain semigroup determinants we also need to compute category determinants (more specifically, groupoid determinants).  If $\mathcal C$ is a category, we write $\mathcal C_0$ for the set of objects, $\mathcal C_1$ for the set of arrows and $\dom,\ran$ for the domain and range maps.
If $\mathcal C$ is a finite category and $K$ is a field, there is a well-known \emph{category algebra} $K\mathcal C$ going back to at least Mitchell~\cite{ringoids}. The algebra has basis $\mathcal C_1$ and the product is defined on the basis by
\[a\cdot b = \begin{cases} ab, & \text{if}\ \dom(a)=\ran(b)\\ 0 , & \text{else.}\end{cases}\]  The algebra $K\mathcal C$ is unital with identity $\sum_{c\in \mathcal C_0}1_c$. Note that the $1_c$ with $c\in \mathcal C_0$ are orthogonal idempotents. From now on, we write $ab$ instead of $a\cdot b$ as no confusion should arrive.  We define the \emph{category determinant} of $\mathcal C$ to be $\theta_{\mathcal C} =\theta_{(\mathbb C\mathcal C,\mathcal C_1)}$.  Notice that $C(\mathcal C) = C(\mathbb C\mathcal C,\mathcal C_1)$ is given by
\[C(\mathcal C)_{a,b} = \begin{cases} x_{ab}, & \text{if}\ \dom(a)=\ran(b)\\ 0, & \text{else.}\end{cases}\]

If a finite monoid $M$ is viewed as a category with a single object, then the determinant  of $M$, viewed as a semigroup or as a category is the same.  More generally, if $\mathcal C$ is a finite category in which every arrow is an endomorphism, then $C(\mathcal C)$ is a block diagonal matrix with the Cayley tables of the endomorphism monoids on the diagonal and so the category determinant of $\mathcal C$ is the product of the semigroup determinants of the endomorphism monoids.

Conversely, category determinants can be computed as variants of the semigroup determinant.  Let $S$ be a semigroup with a zero element $z$; so $zS=\{z\}=Sz$.   Then the \emph{contracted semigroup algebra} of $S$ is $K_0S=KS/Kz$;  note that $Kz$ is a one-dimensional two-sided ideal.  We can view $K_0S$ as having basis the nonzero elements of $S$ and multiplication extending that of $S$, but where we identify the zero of the semigroup with the zero of the algebra. We put $\til C(S) = C(\mathbb C_0S,S\setminus \{z\})$ and $\til \theta_S =\det \til C(S)$ and call it the \emph{contracted semigroup determinant} of $S$ (Wood~\cite{woodsemigroup} uses the term ``reduced'').  We often put $\til X= X_{S\setminus \{z\}}$ if $S$ is understood.  Note that \[\til C(S)_{s,t}= \begin{cases} x_{st}, & \text{if}\ st\neq z\\ 0,  & \text{else.} \end{cases}\]

 For example, if $\mathcal C$ is a finite category, we can build a semigroup $S=\mathcal C_1\cup \{z\}$ by defining $ab=z$ whenever $\dom(a)\neq \ran(b)$ and $zS=\{z\}=Sz$ (and, of course, we keep the multiplication of composable arrows).  Then $\til C(S)= C(\mathcal C)$, as is easily checked, and so $\theta_{\mathcal C}=\til \theta_S$. 

We shall also need later  twisted (contracted) semigroup algebras.  If $S$ is a semigroup, a \emph{twisted semigroup algebra} of $S$ over a field $K$ is a $K$-algebra with a basis $\{\ov s\mid s\in S\}$ in bijection with $S$ with the property that $\ov s\cdot \ov t=c(s,t)\ov{st}$ for some $c(s,t)\in K^\times$.  We sometimes write $K(S,c)$ for the twisted semigroup algebra with twist $c$.  The mapping $c\colon S\times S\to K^\times$ is a $2$-cocycle.   If $S$ is a semigroup with zero $z$, we call a $K$-algebra  with basis $\{\ov s\mid s\in S\setminus \{z\}\}$ in bijection with the nonzero elements of $S$ a \emph{twisted contracted semigroup algebra} of $S$ if
\[\ov s\cdot \ov t=\begin{cases}c(s,t)\ov {st}, & \text{if}\ st\neq z\\ 0, & \text{else}\end{cases}\] where $c(s,t)\in K^\times$ whenever $st\neq z$.  We write $K_0(S,c)$ for this algebra, but note that $c$ is only a partial mapping and hence is slightly more complicated to axiomatize than a standard $2$-cocycle (cf.~\cite{ClarkTwist}). 
Of course, when the twist is trivial we recover the usual notions of semigroup algebra and contracted semigroup algebra.  We will write $\theta_{S,c}$ (respectively, $\til \theta_{S,c}$) for the paratrophic determinant of a twisted (contracted) semigroup algebra with respect to the distinguished basis.  If $M$ is a monoid, then by a twisted (contracted) monoid algebra, we mean a twisted  (contracted) semigroup algebra for which  $\ov 1$ is the identity of the algebra, that is, the cocycle $c$ is normalized in the sense that $c(1,m)=1=c(m,1)$ for all $m$ (in the contracted case with $m\neq z$).  Throughout the text, we will identify the distinguished basis of a twisted semigroup algebra with the semigroup and write $s$ instead of $\ov s$.  Hopefully, no confusion will arise.

 To compute paratrophic determinants, we often use how they behave under direct product and isomorphism.  If $(A_1,B_1),\ldots, (A_n,B_n)$ are based algebras, we put \[(A_1,B_1)\times\cdots\times (A_n,B_n)=(A_1\times\cdots \times A_n, B_1\uplus\cdots\uplus B_n)\] where the union is disjoint.  Here we view $A_1\times\cdots\times A_n$ as an internal direct sum of the $A_i$ with the rule that $A_iA_j=0$  when $i\neq j$.

\begin{Prop}\label{p:direct.prod}
 If $(A_1,B_1),\ldots, (A_n,B_n)$ are based algebras, then \[\theta_{(A_1,B_1)\times\cdots\times (A_n,B_n)} = \theta_{(A_1,B_1)}\cdots\theta_{(A_n,B_n)}.\]
\end{Prop}
\begin{proof}
Since $B_i^2\subseteq A_i$ and $B_iB_j=0$ when $i\neq j$, we have that the matrix $C((A_1,B_1)\times\cdots\times (A_n,B_n))$ is block diagonal with diagonal blocks $C(A_1,B_1),\ldots, C(A_n,B_n)$.  The result follows.
\end{proof}

A more complicated situation is the behavior of the paratrophic determinant under isomorphism (or change of basis).  It turns out that the determinant changes by a linear homogeneous change of variables and multiplication by a nonzero constant; thus the general structure of its factorization into irreducible polynomials is independent of the basis. If $h\colon R\to R'$ is a ring homomorphism, we also use $h$ for the induced homomorphism of matrix algebras that applies $h$ entrywise.

\begin{Thm}\label{t:cvt}
Let $(A,B)$ and $(A',B')$ be based algebras and let $f\colon A\to A'$ be a $\mathbb C$-algebra homomorphism.  Let $P$ be the $B'\times B$ matrix of $f$ with respect to the bases $B$ and $B'$.  Let $\til f\colon \mathbb C[X_B]\to \mathbb C[X_{B'}]$ be the homomorphism $x_b\mapsto \sum_{b'\in B'}P_{b',b}x_{b'}$ induced by $f$ (note that $\til f(x_b)$ is a linear homogeneous polynomial).
Then $\til f (C(A,B)) = P^TC(A',B')P$.  Therefore, if $f$ is an isomorphism, $\theta_{(A,B)} = \det(P)^2\til f\inv(\theta_{(A',B')})$.
\end{Thm}
\begin{proof}
This is a computation.  Let $b_1,b_2\in B$.  Then we have
\begin{align*}
f(b_1) &= \sum_{b_1'\in B'} P_{b_1'b_1}b_1'\\
f(b_2) &= \sum_{b_2'\in B'} P_{b_2'b_2}b_2'.
\end{align*}
Since $f$ is a homomorphism,
\begin{equation}\label{eq:applyf}
f(b_1b_2) =f(b_1)f(b_2) = \sum_{b_1',b_2',b'\in B'} P_{b_1',b_1}P_{b_2',b_2}c_{b',b_1',b_2'}b'.
\end{equation}
On the other hand, $b_1b_2 = \sum_{b\in B}c_{b,b_1,b_2}b$ and so
\begin{equation}\label{eq:applyf.2}
f(b_1b_2) = \sum_{b\in B, b'\in B'}c_{b,b_1,b_2}P_{b',b}b'.
\end{equation}
Comparing \eqref{eq:applyf} and \eqref{eq:applyf.2}, we deduce that
\begin{equation}\label{eq:applyf.3}
\sum_{b\in B}c_{b,b_1,b_2}P_{b',b} = \sum_{b_1',b_2'\in B'} P_{b_1',b_1}P_{b_2',b_2}c_{b',b_1',b_2'}
\end{equation}
for all $b_1,b_2\in B$ and $b'\in B'$.

Therefore, we have
\begin{align*}
\til f(C(A,B)_{b_1,b_2}) &= \sum_{b\in B,b'\in B'} c_{b,b_1,b_2}P_{b',b}x_{b'}\\ & =\sum_{b_1',b_2',b'\in B'} P_{b_1',b_1}P_{b_2',b_2}c_{b',b_1',b_2'}x_{b'}
\\ &= (P^TC(A',B')P)_{b_1,b_2}
\end{align*}
where the second equality used \eqref{eq:applyf.3}.  The final statement is immediate.
\end{proof}

As a first application, we relate the contracted semigroup determinant and semigroup determinant of a semigroup $S$ with zero.  This is essentially in~\cite{woodsemigroup}.

\begin{Prop}\label{p:backnforth}
Let $S$ be a semigroup with zero element $z$.  Then $f\colon \mathbb CS\to \mathbb C_0S\times \mathbb Cz$  sending $s\mapsto (s,z)$ for $s\in S\setminus \{z\}$ and $z\mapsto (0,z)$ is a $\mathbb C$-algebra isomorphism.       Put $y_s=x_s-x_z$ for $s\neq z$ and let $Y=\{y_s\mid s\in S\setminus \{z\}\}$.
Then $\theta_S(X)= x_z\til \theta_S(Y)$.  Hence $\til \theta_S(\til X)$ is obtained from $\theta_S(X)/x_z$ by replacing $x_z$ by $0$.
\end{Prop}
\begin{proof}
Clearly, $f$ is a vector space isomorphism and it is straightforward to verify that it is a multiplicative homomorphism.  If we use $S\setminus \{z\}$ and $z$ as the basis $B$ for $A=\mathbb C_0S\times \mathbb Cz$, then $\theta_{(A,B)} = x_z\til\theta_S$ by Proposition~\ref{p:direct.prod}. The matrix of $f$ with respect to $S$ and $B$ is lower triangular with ones on the diagonal if we place $z$ last in both orderings and hence has determinant one.   Since $f\inv\colon A\to \mathbb C_0S$ sends $s$ to $s-z$ for $s\in S\setminus \{z\}$ and $z$ to $z$, Theorem~\ref{t:cvt} yields $\theta_S= x_z\til \theta_S(Y)$, as required.
\end{proof}

To deal with semisimple algebras, we start with the case of a matrix algebra. Note that matrix algebras are special cases of contracted semigroup algebras because if $B_n$ is the semigroup of $n\times n$ matrix units, together with $0$, then $\mathbb C_0B_n\cong M_n(\mathbb C)$ as the matrix units form a basis for $M_n(\mathbb C)$.

\begin{Prop}\label{p:matrix.alg}
Let $A=M_n(\mathbb C)$ and let $B$ be the basis of matrix units $E_{ij}$.  Let us write $x_{ij}$ instead of $x_{E_{ij}}$.  Then $\theta_{(A,B)} =(-1)^{\binom{n}{2}}(\det [x_{ij}])^n$ and $\det [x_{ij}]$ is an irreducible homogeneous polynomial of degree $n$.
\end{Prop}
\begin{proof}
Note that \[C(A,B)_{E_{ij},E_{kl}} = \begin{cases}x_{il}, & \text{if}\ j=k\\ 0, & \text{else.}\end{cases}\]   Order the basis by fixing $j$ and letting $i$ vary, i.e., $E_{11}, E_{21},\ldots, E_{1n},\ldots E_{nn}$, and permute the columns of $C(A,B)$ according to $E_{ij}\mapsto E_{ji}$ to obtain a new matrix $C$.  Note that $\det C = (-1)^{\binom{n}{2}}\theta_{(A,B)}$.  Moreover, $C$ is $n\times n$ block diagonal with the matrix $[x_{ij}]$ on each diagonal block.  The  first statement follows.

For the second statement, substitute $x_{ii}$ by $t$ for $1\leq i\leq n$, $x_{i+1,i}$ by $-x_i$, for $1\leq i\leq n-1$, $x_{1,n}$ by $-x_n$ and the remaining entries by $0$.  Then $\det [x_{ij}]$ specializes to $t^n-x_1\cdots x_n$, which is irreducible in $\mathbb C[x_1,\ldots, x_n,t]$ by Eisenstein's criterion.  Since the irreducible factors of $\det [x_{ij}]$ are homogeneous, the above substitution cannot send any factor to a non-zero constant.  We conclude that $\det [x_{ij}]$ is irreducible.
\end{proof}

As a corollary we may compute the paratrophic determinant for a semisimple algebra with respect to a basis of matrix units.  We can then apply Theorem~\ref{t:cvt} to recover Frobenius's result for group determinants.  If $(A,B)$ is a based algebra and $\rho\colon A\to M_d(\mathbb C)$ is a representation, then $\chi\colon B\to \mathbb C$ defined by $\chi(b)=\mathrm{tr}(\rho(b))$ is called the \emph{character} of $\rho$ (with respect to $B$).  Of course $\chi$ extends linearly to $A$ and just computes $\mathrm{tr}(\rho(a))$, so the basis $B$ is not really important.

\begin{Cor}\label{c:ss.alg}
Let $A$ be a semisimple algebra with basis $B$.  Suppose that $A$ has $r$ irreducible representations $\rho^{(1)},\ldots, \rho^{(r)}$ (up to equivalence) of degrees $d_1,\ldots, d_r$ and let $\chi_1,\ldots,\chi_r$ be the corresponding characters.  Then $\theta_{(A,B)} = \prod_{i=1}^r P_i^{d_i}$ where $P_i$ is a homogeneous irreducible polynomial of degree $d_i$.  Moreover, $\chi_i$ can be recovered from $P_i$ and the $P_i$ are distinct.
\end{Cor}
\begin{proof}
By Wedderburn-Artin theory, there is an isomorphism \[\rho\colon A\to M_{d_1}(\mathbb C)\times \cdots\times M_{d_r}(\mathbb C)\] given by $\rho(a) = (\rho^{(1)}(a),\ldots, \rho^{(r)}(a))$.   Take as a basis for $\prod_{i=1}^r M_{d_i}(K)$ the $E^{(k)}_{ij}$, with $1\leq i,j\leq d_k$ and $1\leq k\leq r$, and write $x_{ij}^{(k)}$ for the corresponding variable.  Note that $\til \rho$ is given by $\til\rho(x_b) = \sum_{i,j,k}\rho^{(k)}_{ij}(b)x^{(k)}_{ij}$.
Then by Proposition~\ref{p:direct.prod}, Theorem~\ref{t:cvt} and Proposition~\ref{p:matrix.alg}, we have that \[\theta_{(A,B)} = c\cdot \prod_{k=1}^r \til \rho\inv(\det [x_{ij}^{(k)}])^{d_k}= c\cdot \prod_{k=1}^r \left(\det [\til \rho\inv(x_{ij}^{(k)})]\right)^{d_k}\] where $c$ is a nonzero constant and each $P_k=\til \rho\inv (\det [x_{ij}^{(k)}])$ is an irreducible homogeneous polynomial of degree $d_k$ with the $P_k$ distinct.  Since the $\til\rho\inv(x^{(k)}_{ij})$, as $i,j,k$ vary form a basis for the space of homogeneous linear polynomials in $\mathbb C[X_B]$, if we fix $b\in B$,  we can find a specialization of $X_B$ so that, for all $k$,   $\til\rho\inv(x^{(k)}_{ii})$ becomes $t-\rho^{(k)}_{ii}(b)$ and $\til\rho\inv(x^{(k)}_{ij})$ becomes $-\rho^{(k)}_{ij}(b)$ for $i\neq j$.  Then   $P_k$ specializes to the characteristic polynomial of $\rho^{(k)}(b)$, for each $k$,  and hence we can recover $\chi_k(b)$ from the coefficient of $t^{d_k-1}$ of  $P_k$ under this specialization.	The final statement follows.  Of course, we can absorb a root of $c$ into $P_1$ to remove the constant.
\end{proof}

\subsection{The theorems of Frobenius and Wilf-Lindstr\"{o}m}
Next we recover the classical result of Frobenius on factoring the group determinant and a result of  Wilf-Lindstr\"{o}m on the determinant of a semilattice~\cite{Wilf,latticedet2}.  These are special cases of the factorization theorem for the semigroup determinant of an inverse semigroup, but we feel it is worth separating these results out since they will more likely be familiar to the reader and will thus better illustrate our techniques.

To make Corollary~\ref{c:ss.alg} explicit for groups, we describe more carefully the Wedderburn isomorphism, i.e., the Fourier transform. Let $G$ be a finite group, $\rho^{(1)},\ldots, \rho^{(r)}$ representatives of the equivalence classes of unitary irreducible representations of $G$ and $d_k$ the degree of $\rho^{(k)}$.
By Maschke's theorem, the group algebra $\mathbb CG$ of a finite group is semisimple, and the number of conjugacy classes of $G$ is $r$.  We retain the notation from the proof of Corollary~\ref{c:ss.alg}.

\begin{Prop}\label{p:reverse}
Let $\rho\colon \mathbb CG\to M_{d_1}(\mathbb C)\times \cdots\times M_{d_r}(\mathbb C)$ be given by $\rho(g) = (\rho^{(1)}(g),\ldots, \rho^{(r)}(g))$ for $g\in G$. Let $P$ be the matrix of $\rho$ with respect to the bases $G$ and $E^{(k)}_{ij}$.
\begin{enumerate}
\item  $\displaystyle{\til\rho\inv (x^{(k)}_{ij})=\dfrac{d_k}{|G|}\sum_{g\in G}\overline{\rho^{(k)}_{ij}}(g)x_g}$.
\item Viewing, the rows and columns of $PP^T$ as indexed by $i,j,k$, \[PP^T_{ijk,i'j'k'}=\begin{cases} \dfrac{|G|}{d_k}, & \text{if}\ i=i', j=j', \rho^{(k')} = \ov {\rho^{(k)}}\\ 0, & \text{else.} \end{cases}\]
\end{enumerate}
\end{Prop}
\begin{proof}
The first item is an immediate consequence of the Fourier inversion theorem~\cite[Proposition~6.2.11]{serrerep}, which shows that the preimage under $\rho$ of $E^{(k)}_{ij}$ is $\dfrac{d_k}{|G|}\sum_{g\in G}\overline{\rho^{(k)}_{ij}}(g)g$.

For the second item, we observe that \[PP^T_{ijk,i'j'k'} = \sum_{g\in G}\rho^{(k)}_{ij}(g)\rho^{(k')}_{i'j'}(g).\] By the orthogonality relations for the matrix entries of irreducible unitary representations~\cite[Section~2.2]{serrerep}, this sum is $0$ unless $i=i'$, $j=j'$ and $\rho^{(k')}$ is the conjugate representation of $\rho^{(k)}$, in which case the result is $|G|/d_k$.
\end{proof}

\begin{Cor}[Frobenius]\label{c:group.case}
Let $G$ be a finite group with $r$ conjugacy classes and let $\rho^{(1)},\ldots, \rho^{(r)}$ be a complete set of irreducible unitary representations of $G$ with $\rho^{(k)}$ of degree $d_k$.  Then \[\theta_G = \pm \prod_{k=1}^r \left(\det \left[\sum_{g\in G}\rho^{(k)}_{ij}(g)x_g\right]\right)^{d_k}\] is the factorization into irreducible polynomials.  The character of $\rho^{(k)}$ is determined by the irreducible polynomial $P_k=\det [\sum_{g\in G}\rho^{(k)}_{ij}(g)x_g]$.
\end{Cor}
\begin{proof}
We retain the notation of Proposition~\ref{p:reverse}. We shall freely use that $d_1^2+\cdots + d_r^2=|G|$. Note that Proposition~\ref{p:reverse}(2) implies that $\det(P)^2 =\det(PP^T) = \pm\dfrac{|G|^{|G|}}{d_1^{d_1^2}\cdots d_r^{d_r^2}}$.  Then applying Theorem~\ref{t:cvt}, Proposition~\ref{p:direct.prod}, Proposition~\ref{p:matrix.alg}, Proposition~\ref{p:reverse}(1) and that $\det [x^{(k)}_{ij}]$ is homogeneous of degree $d_k$, we see that \[\theta_G = \pm \left(\det \left[\sum_{g\in G}\rho^{(k)}_{ij}(g)x_g\right]\right)^{d_k}\] is the factorization is the factorization into irreducibles.   The second statement follows by computing $\chi(g)$ with $g\neq 1$ by looking at the coefficient of $t^{d_k-1}$ after substituting into $P_k$ the values $x_1=t$, $x_g=-1$ and setting all other variables to $0$ (and, of course, $\chi(1)=\deg P_k$).
\end{proof}

The special case where $G$ is abelian was first obtained by Dedekind.  In this case, the formula simplifies to $\theta_G = \pm\prod_{\chi\in \wh G}\left(\sum_{g\in G}\chi(g)x_g\right)$ where $\wh G$ is the set of characters of $G$.

If $P$ is a finite poset, the \emph{zeta-function} $\zeta_P$ is the $P\times P$ integer matrix with
\[\zeta_P(x,y) = \begin{cases} 1, & \text{if}\ x\leq y\\ 0, & \text{else.}\end{cases}\]  Note that $\zeta$ is upper triangular with ones on the diagonal with respect to any linear order extending $P$ and hence is unimodular.  Therefore, $\zeta_P$ has an inverse over the integers called the \emph{M\"obius} function $\mu_P$ of $P$. We shall drop the subscript $P$ when the poset $P$ is clear from context.

Solomon proved the following theorem (see~\cite[Theorem~1]{Burnsidealgebra} and the remark following it).

\begin{Thm}[Solomon]
Let $L$ be a finite meet semilattice, made into a semigroup via meet.  Then $Z\colon \mathbb CL\to \mathbb C^L$ given by $Z(a) = \sum_{b\leq a}\delta_b$ is an isomorphism where $\delta_b$ is the indicator function of $\{b\}$.	
\end{Thm}

Notice that if $A=\mathbb C^L$ and $B=\{\delta_a\mid a\in L\}$, then $C(A,B)$ is the diagonal matrix with the $x_{\delta_a}$ on the diagonal.
Since the matrix of $Z$  with respect to the bases $L$ and $B$ is $\zeta_L$ (after making the obvious identification of $L$ and $B$), which is unimodular, and the inverse of $\zeta_L$ is $\mu_L$, whence $\til\zeta_L\inv (x_{\delta_a}) = \sum_{b\leq a}\mu_L(b,a)x_b$, we immediately obtain the Wilf-Lindstr\"om theorem~\cite{Wilf,latticedet2} from Theorem~\ref{t:cvt} and Solomon's theorem.

\begin{Thm}[Wilf-Lindstr\"{o}m]
Let $L$ be a finite semilattice.  Then \[\theta_L = \prod_{a\in L}\left(\sum_{b\leq a}\mu_L(b,a)x_b\right).\]	
\end{Thm}

As a corollary, we give the computation for the Smith determinant.   As usual,  $\phi$ denotes Euler's totient function. Recall that $n=\sum_{d\mid n}\phi(n)$ by counting elements of $\mathbb Z/n\mathbb Z$ by the cardinality of the subgroup they generate.  Hence, by M\"obius inversion, $\phi(n) = \sum_{d\mid n}\mu(n/d)d$ where $\mu$ is the number theoretic M\"obius function.  Also, we recall~\cite[Page~119]{Stanley} that for the divisibility order, $\mu(a,b) = \mu(b/a)$; this is from whence the M\"obius function of a poset gets its name.

\begin{Cor}[Smith's determinant]
The determinant of the $n\times n$ matrix $A$ with $A_{ij} = \gcd(i,j)$ is $\phi(1)\phi(2)\cdots\phi(n)$.
\end{Cor}
\begin{proof}
Note that $L=\{1,\ldots, n\}$ is a meet semilattice under the divisibility ordering with $\gcd$ as the meet and $A$ is obtained from $C(L)$ by substituting $x_i$ by $i$.  Thus the Wilf-Lindstr\"om theorem yields \[\det A=\prod_{i=1}^n\left(\sum_{d\mid i}\mu(d,i)d\right)= \phi(1)\cdots\phi(n)\] by the above discussion.
\end{proof}

\section{Inverse semigroups and groupoids}

We factor in this section the semigroup determinant of an inverse semigroup.  Inverse semigroups have semisimple algebras over the complex numbers, so Corollary~\ref{c:ss.alg} applies, but we would like to be more explicit, as in the group case.   Note that there are nice Fourier transforms for inverse semigroups~\cite{Malandro2}.  Here we show that the inverse semigroup determinant can be computed as a groupoid determinant.  The orthogonality relations and Fourier inversion theorems can be made to work for finite groupoids (where unitary representations are replaced by an appropriate groupoid analogue) but I do not know a reference for this folklore stuff and so we do not make things quite as explicit as could be done, except in a remark.  However, we do make the factorization completely explicit in the commutative case.

\subsection{Inverse semigroups}
An \emph{inverse semigroup} $S$ is a semigroup such that, for all $s\in S$, there is a unique element $s^*\in S$ such that $ss^*s=s$ and $s^*ss^*=s^*$.  For example, groups are inverse semigroups (where $g^*=g^{-1}$) and meet semilattices are examples (where $x^*=x$).  Thus a factorization result for the semigroup determinant of an inverse semigroup yields a simultaneous generalization of Frobenius's result for groups and the Wilf-Lindstr\"om result for semilattices.  Another important example of an inverse semigroup is the \emph{symmetric inverse monoid} $I_X$ of all partial bijections of a set $X$.  When $|X|=n$, then $I_X$ can be identified with the inverse monoid of all $n\times n$ $0/1$-matrices with at most one $1$-entry in any row or column; this latter monoid is known in some circles as the rook monoid~\cite{Solomonrook}, since such matrices correspond to legal rook placements on an $n\times n$ chessboard (where the $1$s are the rooks).  A good reference for inverse semigroup theory is the book of Lawson~\cite{Lawson}.

Let $E(S)$ denote the set of idempotents of $S$.  In an inverse semigroup, the idempotents commute and hence form a commutative subsemigroup~\cite{Lawson}.
On any inverse semigroup $S$, there is a natural partial order, compatible with multiplication, given by $s\leq t$ if $s=te$ for some $e\in E(S)$. 
See~\cite{Lawson} for details. 

To any inverse semigroup is associated a groupoid $\mathcal G(S)$.  One puts $\mathcal G(S)_0=E(S)$ and $\mathcal G(S)_1=S$.  The domain and range functions are defined by $\dom (s)=s^*s$ and $\ran (s) =ss^*$.  The identities are the elements of $E(S)$, the product in $\mathcal G(S)$ is just the restriction of the product in $S$ to composable pairs and the groupoid inversion is  $s\mapsto s^*$.  The fact that this is a groupoid can be found in~\cite{Lawson} or~\cite[Section 9.1]{repbook}.  
Notice that $\mathcal G(S)$ is finite if and only if $S$ is finite.  It is well known~\cite{Lawson}, that the idempotents of $S$ are central if and only if $\dom(s)=\ran(s)$ for all $s\in S$; such inverse semigroups are called Clifford semigroups~\cite{Clifford}.

If $e\in E(S)$, the automorphism group $G_e=\mathcal G(S)(e,e)$ is called the \emph{maximal subgroup} of $G$ at $e$.  It is the group of units of the monoid $eSe$.

The author proved the following theorem in~\cite{mobius2} (see also~\cite{mobius1,repbook}), generalizing Solomon's theorem for semilattices~\cite{Burnsidealgebra}. 

\begin{Thm}\label{t:algebra.isom}
Let $S$ be a finite inverse semigroup and $K$ a field.  Then the mapping $Z\colon KS\to K\mathcal G(S)$ given by $Z(s) = \sum_{t\leq s}t$ for $s\in S$ is an isomorphism of $K$-algebras. In particular, $KS$ is unital.
\end{Thm}

The idea to compute the semigroup determinant of an inverse semigroup $S$ is to first use Theorem~\ref{t:algebra.isom} to reduce to the case of a groupoid, which is then a semisimple algebra and so can be handled by Corollary~\ref{c:ss.alg}.

\subsection{Factoring the groupoid determinant}
To factor the groupoid determinant, we use Corollary~\ref{c:ss.alg}.
First we comment on the well-known structure of a groupoid algebra; see~\cite[Theorem~8.15]{repbook} or~\cite{ringoids}.

\begin{Thm}\label{t:gpd.as.matrices}
Let $\mathcal G$ be a finite groupoid and let $c_1,\ldots, c_k\in \mathcal G_0$ be a complete set of representatives of the isomorphism classes of objects.  Suppose that the class of $c_i$ has $n_i$ elements and let $G_i=\mathcal G(c_i,c_i)$ be the automorphism group at $c_i$.  Then, for any field $K$, there is an isomorphism $K\mathcal G\cong \prod_{i=1}^k M_{n_i}(KG_i)$.  In particular, $K\mathcal G$ is a semisimple algebra if and only if the characteristic of $K$ does not divide the order of any of the groups $G_i$.
\end{Thm}

The isomorphism in Theorem~\ref{t:gpd.as.matrices} is explicit.
Our takeaway from this is that $\mathbb C\mathcal G$ is semisimple.   For example, if $L$ is a meet semilattice, then $\mathcal G(L)$ consists of only identity morphisms and so $K\mathcal G(L)\cong K^L$.

\begin{Thm}\label{t:gpd.determine}
Let $\mathcal G$ be a finite groupoid.  Then $\theta_{\mathcal G} = \prod_{i=1}^r P_i^{d_i}$ where the $P_i$ are distinct irreducible homogeneous polynomials of degree $d_i$.   Here $r$ is the number of isomorphism classes of simple $\mathbb C\mathcal G$-modules, the $d_i$ are the dimensions of the simple modules and $P_i$ determines the character of the corresponding simple module.
\end{Thm}
\begin{proof}
Since $\mathbb C\mathcal G$ is semisimple by Theorem~\ref{t:gpd.as.matrices}, this is immediate from Corollary~\ref{c:ss.alg}.
\end{proof}

\begin{Rmk}\label{rmk.gpd}
In fact, the exact analogue of Corollary~\ref{c:group.case} holds for groupoids.  That is, if $\rho^{(1)},\ldots, \rho^{(r)}$ are representatives of the equivalence classes of irreducible representations of $\mathbb C\mathcal G$, then one can take $P_k = \det[\sum_{g\in \mathcal G_1}\rho_{ij}^{(k)}(g)x_g]$.  Indeed, if we consider the generic element $a=\sum_{g\in \mathcal G_1}x_gg\in \mathbb C(X_{\mathcal G_1})\mathcal G$, then one can easily verify that the matrix of $a$ under the left regular representation $\lambda$ (with respect to the basis $\mathcal G_1$) is given by
\[\lambda(a)_{gh} = \begin{cases}x_{gh\inv}, & \text{if}\ \dom(g)=\dom(h)\\ 0, & \text{else.}\end{cases}\]  Therefore, $\lambda(a)$ is obtained from $C(\mathcal G)$ by permuting the columns according to $h\mapsto h\inv$ and so $\theta_{\mathcal G} = \pm \det \lambda(a)$.  If $\rho^{(k)}$ has degree $d_k$, then $\rho^{(k)}$ (with the scalars extended to $\mathbb C(X_{\mathcal G_1})$) appears $d_k$ times in the left regular representation and so
\[\theta_{\mathcal G} = \pm\prod_{k=1}^r \left(\det\left[\sum_{g\in \mathcal G_1}\rho_{ij}^{(k)}(g)x_g\right]\right)^{d_k}.\]
Moreover, $\det[\sum_{g\in \mathcal G_1}\rho_{ij}^{(k)}(g)x_g]$ is irreducible by~\cite[Section~23]{Noether} (or by the same argument as in the group case~\cite{ConradRep}).
\end{Rmk}

We remark that if all arrows of the groupoid $\mathcal G$ are automorphisms (i.e, each isomorphism class is a singleton), then $C(\mathcal G)$ is block diagonal with diagonal blocks $C(G_c)$ where $G_c$ is the group $\mathcal G(c,c)$.  Therefore,  $\theta_{\mathcal G} = \prod_{c\in \mathcal G_0} \theta_{G_c}$ and so  no additional work beyond Frobenius is actually required in this case.

\subsection{Factoring the semigroup determinant of an inverse semigroup}

We are now ready for the main theorem of this section.  Let $S$ be a finite inverse semigroup.  Let $\sim$ be the smallest equivalence relation on $S$ so that $ab\sim ba$ for all $a,b\in S$.  For a group, this relation is conjugacy and so we shall call it conjugacy here.  For inverse semigroups, it coincides~\cite{Mazorchuk} with the notion of generalized conjugacy in the sense of McAlister~\cite{McAlisterCharacter} or Rhodes and Zalcstein~\cite{RhodesZalc}, as discussed in~\cite[Section~7.1]{repbook}, and so the number of simple $\mathbb CS$-modules is the number of conjugacy classes.

\begin{Thm}\label{t:inverse.sgp.case}
Let $S$ be a finite inverse semigroup.  Then $\theta_S = \prod_{i=1}^r P_i^{d_i}$ where the $P_i$ are distinct irreducible homogeneous polynomials of degree $d_i$.  Here $r$ is the number of conjugacy classes of $S$, the $d_i$ are the degrees of the irreducible representations of $S$ and the $P_i$ determine the irreducible characters of $S$.  More precisely, for $s\in S$, put $y_s=\sum_{t\leq s}\mu_S(t,s)x_t$ and put $Y=\{y_s\mid s\in S\}$.  Then $\theta_S(X) = \theta_{\mathcal G(S)}(Y)$.
\end{Thm}
\begin{proof}
This follows from Theorem~\ref{t:algebra.isom}, Theorem~\ref{t:cvt} and Theorem~\ref{t:gpd.determine} once we make the following observation.  The matrix of $Z$ from Theorem~\ref{t:algebra.isom} with respect to the bases $S$ for $\mathbb CS$ and $\mathbb C\mathcal G(S)$ is the zeta-function of $S$, which is unimodular.  The inverse isomorphism has matrix $\mu_S$ and hence is given by $s\mapsto \sum_{t\leq s}\mu_S(t,s)t$.  Therefore, $\til Z\inv(x_s) =\sum_{t\leq s}\mu_S(t,s)x_t =y_s$.
\end{proof}

Remark~\ref{rmk.gpd} can be used to write out more explicitly $\theta_S$.

In the case that $S$ is commutative, $\mathcal G(S)$ is just the disjoint union of the abelian groups $G_e$ with $e\in E(S)$ and we can apply Dedekind's factorization theorem for the determinant of an abelian group.

\begin{Thm}\label{t:comm.case}
Let $S$ be a finite commutative inverse semigroup.  Then \[\theta_S = \pm\prod_{e\in E(S)}\prod_{\chi\in \wh{G_e}}\left(\sum_{g\in G_e}\sum_{s\leq g} \chi(g)\mu_S(s,g)x_s\right).\]  In particular, $\theta_S$ factors into distinct linear factors and the factors are determined by the characters of the maximal subgroups of $S$ (or equivalently of $S$) and the M\"obius function of $S$.  More generally if $S$ is an inverse semigroup with central idempotents, then $\theta_S = \prod_{e\in E(S)} \theta_{G_e}(Y_e)$ where $Y_e = \{y_g\mid g\in G_e\}$ and $y_g=\sum_{s\leq g} \mu_S(s,g)x_s$.
\end{Thm}

When $S$ is an abelian group, this reduces to Dedekind's theorem and when $S$ is a meet semilattice (and so all $G_e$ are trivial), this reduces to the Wilf-Lindstr\"om theorem.

\section{The semigroup determinant of a finite Frobenius ring}\label{s:Frob.ring}
In this section we show that if $R$ is a finite Frobenius ring and $K$ is a field whose characteristic does not divide the characteristic of $R$, then $KR$ is a Frobenius algebra where we take the monoid algebra with respect to the multiplicative monoid of $R$. From the viewpoint of this paper, the upshot is that the semigroup determinant of $R$ does not vanish.  But more importantly, as a consequence, we obtain a new, simpler proof of the theorem of Kov\'acs~\cite{Kovacs} that $KM_n(F)$ is isomorphic to a direct product of matrix algebras over group algebras of general linear groups  for any finite field $F$ and field $K$ whose characteristic is different from that of $F$ (the characteristic zero case is due to Ok\'ninski and Putcha~\cite{putchasemisimple}).

A finite unital ring $R$ is \emph{Frobenius} if it is self-injective and $R/J(R)$ is isomorphic to the left socle of $R$ as left $R$-modules (where $J(R)$ is the Jacobson radical of $R$).  A Frobenius algebra over a finite field is Frobenius in this sense.  In particular, $M_n(F)$ is a Frobenius ring for any finite field $F$.  It is well known that a finite commutative ring $R$ is Frobenius if and only if it is a direct product of finite local rings, each having a unique minimal ideal~\cite{wood}.  In particular, any finite principal ideal ring (like $\mathbb Z/n\mathbb Z$) is Frobenius.

If $R$ is a ring and $A$ is an abelian group, the group  $\Hom(R,A)$ of homomorphisms from the additive group of $R$ to $A$ is an $R$-$R$-bimodule via $(fr)(r') = f(rr')$ and $(rf)(r') = f(r'r)$; moreover, the bimodule structure is functorial in $A$.
If $K$ is a field, we shall call a character $\chi\colon R\to K^\times$ of the additive group of $R$ an \emph{additive character} and we put $\wh R_K=\Hom(R,K^\times)$, which is then an $R$-$R$-bimodule with the above bimodule structure.  An additive character $\chi$ is called a \emph{left (respectively, right) generating character} if it generates $\wh R_K$ as a left (respectively, right) $R$-module.   We write $\wh R$ for $\wh R_{\mathbb C}$.  It is known~\cite[Theorem~4.3]{wood} that any left generating character  of $\wh R$ is a right generating character (and vice-versa) and so we may omit the left/right descriptives in this setting. If $K$ is a field, then $\mu_{n,K}$ will denote the group of $n^{th}$-roots of unity in $K$.

\begin{Prop}\label{p:indep.field}
Let $R$ be a finite ring and $K$ an algebraically closed field whose characteristic does not divide the characteristic of $R$.  Then $\wh R\cong \wh R_K$ as an $R$-$R$-bimodule.
\end{Prop}
\begin{proof}
Let $n$ be the characteristic of $R$.  Then by assumption, $\mu_{n,K}\cong \mathbb Z/n\mathbb Z\cong \mu_{n,\mathbb C}$ since both fields contain a primitive $n^{th}$-root of unity.  Moreover, since $nR=0$, we have $\wh R_K=\Hom(R,\mu_{n,K})\cong \Hom(R,\mu_{n,\mathbb C})= \wh R$
 as $R$-$R$-bimodules by functoriality of $\Hom$.
\end{proof}

The following theorem is due to Wood~\cite[Theorem~3.10]{wood}.

\begin{Thm}\label{t:generating.character}
A finite ring $R$ is Frobenius if and only if $\wh R$ has a generating character, i.e., $\wh R$ is a cyclic $R$-module (on either side).
\end{Thm}

Observe that if $p$ is a prime and $R$ is a finite unital ring, then $p$ divides the characteristic $n$ of $R$ if and only if it divides $|R|$.  Indeed, since $n$ is the additive order of $1$, clearly $n\mid |R|$, and so the forward implication is trivial.  On the other hand, if $p\mid |R|$, then the additive group of $R$ has an element $r$ of order $p$.  But, we also have $nr=0$, and so $p\mid n$.

We can now prove our main result of this section.

\begin{Thm}\label{t:frob.ring}
Let $R$ be a finite Frobenius ring and $K$ a field whose characteristic does not divide that of $R$.  Then $KR$ is a Frobenius algebra, where we take the monoid algebra with respect to the multiplicative monoid of $R$.
\end{Thm}
\begin{proof}
By~\cite[Exercise~3.16]{Lam2} (or~\cite[Chapter~16, Lemma~19]{oknisemigroupalgebra}), we may assume without loss of generality that $K$ is algebraically closed.
By Proposition~\ref{p:indep.field}, we have that $\wh R_K\cong \wh R$ as an $R$-$R$-bimodule.  By Theorem~\ref{t:generating.character}, there is an additive character of $\wh R$ that generates $\wh R$ as both a left and a right $R$-module.  Hence there is an additive character $\lambda\in \wh R_K$ that generates it as both a left and right $R$-module.  We can extend any additive character $\chi\colon R\to K^\times$ to a $K$-linear mapping $\chi\colon KR\to K$.  Moreover, since the characteristic of $K$ does not divide the characteristic of $R$ (and hence $|R|$), we have that the group algebra of the additive group of $R$ over $K$ is semisimple and so the functionals $\chi\colon KR\to K$ with $\chi \in\wh R_K$ separate points (since $KR$ has the same underlying $K$-vector space as the group algebra of the additive group).  Therefore, if $0\neq a\in KR$, then we can find $\chi\in \wh R_K$ with $\chi(a)\neq 0$.  Since $\lambda$ is generating, we can find $r,r'\in R$ with $r\lambda = \chi=\lambda r'$.  Then $0\neq \chi(a) = \lambda(ar)=\lambda (r'a)$.  We conclude that $\ker \lambda$ contains no nonzero left or right ideal of $KR$, and so $KR$ is Frobenius.
\end{proof}

It follows that $K[\mathbb Z/n\mathbb Z]$ is Frobenius whenever the characteristic of $K$ does not divide $n$, where we take the monoid algebra with respect to the multiplicative monoid of the ring.

As a corollary, we obtain the following result on the nonvanishing of the semigroup determinant of a finite Frobenius ring.

\begin{Cor}\label{c:frob.ring}
Let $R$ be a finite Frobenius ring.  Then the semigroup determinant of the multiplicative monoid of $R$  is nonzero.
\end{Cor}
\begin{proof}
This is immediate from Theorems~\ref{t:vanish.frob} and~\ref{t:frob.ring}.
\end{proof}

Next we recover Kov\'acs's theorem~\cite{Kovacs}.  Recall that the $q$-binomial coefficient $\binom{n}{r}_q$ counts the number of $r$-dimensional subspaces of an $n$-dimensional vector space over a $q$-element field.

We need to recall Green's relations~\cite{Green}.  If $S$ is a semigroup, then two elements are \emph{$\mathscr J$-equivalent} if they generate the same two-sided ideal, \emph{$\mathscr L$-equivalent} if they generate the same left ideal and \emph{$\mathscr R$-equivalent} if they generate the same right ideal.   If $e$ is an idempotent of a semigroup $S$, the maximal subgroup of $S$ at $e$ is the group of units of the monoid $eSe$.  In a finite semigroup, $\mathscr J$-equivalent idempotents have isomorphic maximal subgroups~\cite{qtheor,repbook}.  It is shown in~\cite[Theorem~15.6]{repbook} (see also~\cite[Chapter~16, Theorem~17]{oknisemigroupalgebra})  that if $M$ is a finite regular monoid and $e_1,\ldots, e_s$ represent the  $\mathscr J$-classes of idempotents of $M$, then $KM$ is Frobenius if and only if $KM\cong \prod_{i=1}^sM_{n_i}(KG_{e_i})$ where $n_i$ is the number of $\mathscr L$-classes in the $\mathscr J$-class of $e_i$.

The multiplicative monoid of $M_n(F)$ (with $F$ a finite field) is regular (all semisimple rings are von Neumann regular). It is well known that two matrices are $\mathscr J$-equivalent if and only if they have the same rank and they are $\mathscr L$-equivalent if and only if they have the same kernel (that is, are row equivalent).   Let \[e_r =\begin{bmatrix} I_r & 0\\ 0&0\end{bmatrix}.\]  Then $e_r$ is an idempotent, $e_rM_n(F)e_r\cong M_r(F)$ and $G_{e_r}\cong \mathrm{GL}_r(F)$.  The idempotents $e_0,\ldots, e_n$ represent all the regular $\mathscr J$-classes and there are $\binom{n}{r}_q$ possible kernels of rank $r$ matrices, and hence this is the number of $\mathscr L$-classes in the $\mathscr J$-class of rank $r$ matrices.  The following theorem was originally proved in~\cite{Kovacs}.

\begin{Thm}[Kov\'acs]
Let $F$ be a finite field with $q$ elements and let $K$ be a field of characteristic different from that of $F$.  Then $KM_n(F)\cong \prod_{r=0}^nM_{\binom{n}{r}_q}(K\mathrm{GL}_r(F))$ and hence is semisimple if and only if the characteristic of $K$ does not divide $|\mathrm{GL}_n(F)|$.
\end{Thm}
\begin{proof}
Since $M_n(F)$ is a Frobenius ring,  $KM_n(F)$ is a Frobenius algebra by Theorem~\ref{t:frob.ring}.  Since $M_n(F)$ is a semisimple ring, it is von Neumann regular.  The result then follows from~\cite[Theorem~15.6]{repbook} (or~\cite[Chapter~16, Theorem~17]{oknisemigroupalgebra}) and the well-known structure of Green's relations on $M_n(F)$ discussed above.
\end{proof}

Let $F$ be a finite field of characteristic $p$, $K$ an algebraically closed field of characteristic different than $p$ and $\omega\in K^\times$ a primitive $p^{th}$-root of unity.  Denote by $\mathbb F_p$ the field of $p$ elements. Then a generating character $\lambda$ of $M_n(F)$ over $K$ is given by \[A\mapsto \omega^{\mathrm{Tr}_{F/\mathbb F_p}(\mathrm{tr}(A))}\] where $\mathrm{Tr}_{F/\mathbb F_p}$ is the trace map of the field extension and $\mathrm{tr}$ is the matrix trace; see~\cite[Example~4.4]{wood} for the case $K=\mathbb C$.  The linear extension of $\lambda$ to $KM_n(F)$ then provides a Frobenius form.

\section{The trouble with nilpotent semigroups}\label{s:nilpotents}
A semigroup $S$ is called \emph{$k$-nilpotent} if it has a zero element and the product of any $k$ elements is zero.  We put $S_m=S^m\setminus S^{m+1}$ for $m\geq 1$.  Notice that any binary operation with a zero element and the property that the product of any three elements is zero is automatically associative. 
So it should be no surprise that the vast majority of associative binary operations on a set are $3$-nilpotent~\cite{Kleitman} and it is generally believed that the proportion of isomorphism classes of $3$-nilpotent semigroups goes to $1$ as the cardinality goes to infinity.
Let $S$ be a nilpotent semigroup with zero $z$ and suppose that $S^{k+1}=\{z\}$ and $S^k\neq \{z\}$; we call $k+1$ the \emph{nilpotency index} of $S$. A nonzero nilpotent semigroup $S$ obviously does not satisfy $S^2=S$, and  so $\til\theta_S=0$ by Corollary~\ref{c:squared.algebra}. 
Therefore, to obtain something interesting, we adjoin an identity.  

Many of the results in this section on algebras of nilpotent semigroups with adjoined identities are inspired by Wenger~\cite{Wenger}, but we add in the possibility of allowing a twist, which will be needed in the next section to compute the semigroup determinant of a commutative semigroup. Moreover, our results refine those of Wenger.  See also~\cite[Chapter~16]{oknisemigroupalgebra}.

If $S$ is a nilpotent semigroup and $M=S\cup \{I\}$ is obtained by adjoining the identity $I$, we put $S_0=\{I\}$.  Observe that any twisted contracted monoid algebra $\mathbb C_0(M,c)$ is a local ring.   In fact, the nonzero elements of $S$ span the nilpotent ideal $\mathbb C_0(S,c)$ and $\mathbb C_0(M,c)/\mathbb C_0(S,c)\cong \mathbb C$.  Therefore, $\mathbb C_0(S,c)$ is the radical and $\mathbb C_0(M,c)$ is local.

Let us call a nonzero element $m\in M$ \emph{left} (respectively, \emph{right}) \emph{annihilating} if $mS=\{z\}$ (respectively, $Sm=\{z\}$).  If $m$ is both left and right annihilating, we say that $m$ is \emph{annihilating}.  For example, if $S$ has nilpotency index $k+1$, then each element of $S_k$ is annihilating.  Note that if $S=\{z\}$, then the identity $I$ is annihilating.

For a Frobenius algebra $A$, the left module $A/\mathrm{rad}(A)$ is isomorphic to the left socle of $A$ (and dually for the right module $A/\mathrm{rad}(A)$ and the right socle); see the proof of~\cite[Theorem~16.21]{Lam2}). In particular, if $A$ is local (and over $\mathbb C$), then the socle must be simple (and $1$-dimensional).

\begin{Prop}\label{p:socle}
Let $S$ be a nilpotent semigroup and $M=S\cup \{I\}$.  Suppose that the twisted contracted monoid algebra $\mathbb C_0(M,c)$ is Frobenius.  Then $M$ has a unique right annihilating element $z'$. Moreover, $z'$ is annihilating and is the unique left annihilating element of $S$,  and the left and right socles of $\mathbb C_0(M,c)$ are both $\mathbb Cz'$.
\end{Prop}
\begin{proof}
We already observed that $\mathbb C_0(M,c)$ is local and that its radical is $\mathbb C_0(S,c)$.  Also,  we have remarked that if $S$ has nilpotency index $k+1$, then $S_k$ consists of annihilating elements.  Suppose that $z'\in M$ is right annihilating.    Since $Sz'=\{z\}$, it follows that $\mathbb C_0(S,c)\mathbb Cz'=0$ and so $\mathbb Cz'$ is contained in the left socle of $\mathbb C_0(M,c)$.  But since $\mathbb C_0(M,c)$ is a local Frobenius algebra over $\mathbb C$, its left socle is one-dimensional.  Thus $\mathbb Cz'$ is the left socle.  If $z''$ is also right annihilating, then $Cz''$ is also the left socle and so $z'=z''$, yielding uniqueness.  In particular, since $M$ has annihilating elements, the unique right annihilating element $z'$ must be annihilating.  A dual argument then shows that $z'$ is the unique left annihilating element of $M$ and $\mathbb Cz'$ is the right socle of $\mathbb C_0(M,c)$.
\end{proof}

The following is a generalization, and more precise version, of a result due to Wenger~\cite{Wenger}; see also~\cite[Chapter~16, Proposition~22]{oknisemigroupalgebra}.

\begin{Thm}\label{t:nilvanish}
Let $M=S\cup \{I\}$ with $S$ a  nilpotent semigroup.   Then the  determinant  $\til\theta_{M,c}$ vanishes identically for any twisted contracted monoid algebra $\mathbb C_0(M,c)$ unless $M$ has a unique annihilating element $z'$.   In this case, define an $M\setminus \{z\}\times M\setminus \{z\}$-matrix $A$ by \[A_{s,t} = \begin{cases}c(s,t), & \text{if}\ st=z'\\ 0, & \text{else.} \end{cases}\]
Then $\til\theta_{M,c} = \det A\cdot x_{z'}^{|S|}$. Note that if $c$ is trivial, then $A$ is a $0/1$-matrix and $\det A$ is an integer.
\end{Thm}
\begin{proof}
By Proposition~\ref{p:socle}, having a unique annihilating element $z'$  is a necessary condition for $\mathbb C_0(M,c)$ to be Frobenius, and hence for the determinant not to vanish by Theorem~\ref{t:vanish.frob}.  Notice that the matrix $A$ is obtained from $C(\mathbb C_0(M,c), M\setminus \{z\})$ by specializing $x_{z'}$ to $1$ and all other variables to $0$.  Hence, if $\til\theta_{M,c}=0$, then $\det A=0$ and so the formula is true. So suppose now that $\til \theta_{M,c}$ is not zero. Then $\mathbb C_0(M,c)$ is Frobenius and by the proof of Theorem~\ref{t:vanish.frob}, we have that the linear extension of $\lambda\colon M\setminus \{z\}\to \mathbb C$ to $\mathbb C_0(M,c)$ gives rise to a nondegenerate bilinear form $(a,b)\mapsto \lambda(ab)$ if and only if substituting $x_m$ by $\lambda(m)$  in $\til\theta_{M,c}$ yields a nonzero number.  Nondegeneracy of the bilinear form associated to $\lambda$ is equivalent to $\ker\lambda$ containing no left ideal~\cite[Theorem~3.15]{Lam2}.  Since $\mathbb Cz'$ is the left socle of $\mathbb C_0(M,c)$ by Proposition~\ref{p:socle}, the latter condition is equivalent to $\lambda(z')\neq 0$. It follows that any substitution of the variables by complex numbers with $x_{z'}$ not sent to zero results in a nonzero value of $\til\theta_{M,c}$.

Let $q$ be an irreducible factor of $\til \theta_{M,c}$.  Then $q$ must be homogeneous and so we can write $q=f_d+f_{d-1}x_{z'}+\cdots+f_0x_{z'}^d$ where $d=\deg q$ and $f_i$ is a homogeneous polynomial of degree $i$ over $\til X\setminus\{x_{z'}\}$. Suppose first that $f_0=0$.  Then since $f_1,\ldots, f_d$ are homogeneous, they vanish when we substitute all the variables in $\til X\setminus \{x_{z'}\}$ by $0$. 
Thus $q(\vec 0,1)=0$, contradicting that $\til\theta_{M,c}$ does not vanish so long as we evaluate $x_{z'}$ at a nonzero element.  It follows that $f_0\neq 0$.  Suppose $f_i\neq 0$ for some $i>0$ and that $f_i(\vec \alpha)\neq 0$ (using $\mathbb C$ is infinite).  Then $h(x_{z'}) = q(\vec \alpha,x_{z'}) = g(x_{z'})+f_0x_{z'}^d$ where $g$ is a nonzero polynomial of degree less than $d$ and $f_0$ is a nonzero constant.  Thus $h$ has a root $a\neq 0$.  But then $\til\theta_{M,c}(\vec \alpha,a)=0$, again a contradiction since we specialized $x_{z'}$ to $a\neq 0$.  Thus $q=f_0x_{z'}^d$, whence by irreducibility $q=f_0x_{z'}$. We deduce that $\til\theta_{M,c} = Kx_{z'}^{|S|}$ for some constant $K\in \mathbb C$.  Specializing $x_{z'}$ to $1$ and all other variables to $0$ gives $\det A=\til\theta_{M,c}(\vec 0,1) = K$ and the result follows.
\end{proof}

The following example will be used later in the proof of Wood's theorem~\cite{woodsemigroup} on commutative chain rings.  Note that $\mathbb C[x]/(x^k)$ is well known to be a Frobenius algebra, but we verify this using Theorem~\ref{t:nilvanish} to demonstrate the technique.

\begin{Prop}\label{p:cyclic.nilpotent}
Let $S=\langle a\mid a^k=a^{k+1}\rangle$ be a cyclic nilpotent semigroup of index $k$ and $M=S\cup \{I\}$.  Then $\til\theta_M  = (-1)^{\binom{k}{2}}x_{a^{k-1}}^k$.
\end{Prop}
\begin{proof}
Note that $a^{k-1}$ is the unique annihilating element of $M$ and so $\til\theta_M = \det A\cdot x_{a^{k-1}}^k$ where
\[A = \begin{bmatrix}0 & 0&\cdots & 1\\ 0 & \cdots & 1 & 0\\ \vdots &1 &\iddots &\\ 1 & 0& \cdots&0\end{bmatrix}\]   by Theorem~\ref{t:nilvanish}.  But this is the permutation matrix of the unique permutation with $\binom{k}{2}$ inversions.  This completes the proof.
\end{proof}

We next show that determining whether the semigroup determinant of a $3$-nilpotent semigroup with adjoined identity vanishes is a tricky business.  We already saw in Theorem~\ref{t:nilvanish} that this reduces to the problem of determining whether a $0/1$-matrix is nonsingular.  We now show that for $3$-nilpotent (commutative) semigroups with adjoined identity, the problem of determining whether $\til \theta_M$ vanishes is equivalent to the problem of determining whether a (symmetric) $0/1$-matrix is nonsingular.

\begin{Prop}\label{p:3nil}
Let $B$ be an $n\times n$ $0/1$-matrix.  Then there is a $3$-nilpotent semigroup $S$ such that if
 $M=S\cup \{I\}$,  then $\til\theta_M = -\det B\cdot x_{z'}^{n+2}$.  Moreover, if $B$ is symmetric, then $M$ is commutative.
\end{Prop}
\begin{proof}
Let $S=\{s_1,\ldots, s_n,z',z\}$ where $s_is_j=z'$ if $B_{ij}=1$ and all remaining products are $z$.  Then $S$ is $3$-nilpotent and is commutative if and only if $B$ is symmetric.  Let $A$ be the matrix from Theorem~\ref{t:nilvanish}.  If we order $M\setminus \{z\}$ as $I,s_1,\ldots, s_m, z'$,  then \[A = \begin{bmatrix} 0 & \cdots & 1\\ \vdots & B & \vdots\\ 1 & \cdots & 0\end{bmatrix} \] (where all omitted entries are $0$),  and so $\det A=-\det B$.  The result follows from Theorem~\ref{t:nilvanish}.
\end{proof}

\section{The curious case of commutative semigroups}
In this section, we factor the semigroup determinant of a commutative semigroup.  Our results will recover Theorem~\ref{t:comm.case} (hence Dedekind and Wilf-Lindstr\"om) and Wood's theorem (generalized from rings to monoids).

According to Corollary~\ref{c:squared.algebra}, a necessary condition for $\theta_S\neq 0$ is that $S^2=S$. Such semigroups are sometimes called  \emph{idempotent semigroups}, although this can be confused with semigroups all of whose elements are idempotents, also known as bands.  In this paper we shall use idempotent semigroup to mean $S^2=S$, as the only bands that we shall encounter are semilattices.  Of course every monoid is idempotent and so is every inverse semigroup.

It is well known that a finite semigroup $S$ satisfies $S=S^2$ if and only if $S=SE(S)S$, where, as usual, $E(S)$ denotes the set of idempotents of $S$ (this follows, for instance, from~\cite[Lemma~4.5.6]{qtheor}).

 A number of our reductions hold for semigroups with central idempotents and so we start there.

\subsection{Idempotent semigroups with central idempotents}

Throughout this subsection we shall assume that $S$ is finite, $S^2=S$ and that
 $E(S)$ is contained in the center  of $S$.  Note that $s=xey$ with $e\in E(S)$ implies $ese=exeye=xey=s$ and so $e$ is a left and right identity for $s$.  Also, since $E(S)$ commutes, it is a meet semilattice with respect to the ordering $e\leq f$ if $ef=e$ (or equivalently, $e\in fE(S)$).  The set of $e\in E(S)$ with $se=s$ (or, equivalently, $ese=s$) is a nonempty subsemigroup of $E(S)$ (by the above discussion) and hence has a unique minimum element that we denote by $s^+$. Note that $e^+=e$ for $e\in E(S)$.   We define an equivalence relation by $s\sim t$ if $s^+=t^+$. We write $\til H_s$ for the equivalence class of $s$ (this is the $\til{\mathscr H}$-class of $s$ in the sense of Fountain \textit{et al}.~\cite{Ltilde}). Note that $\til H_s\subseteq s^+Ss^+=s^+S=Ss^+$.

\begin{Prop}\label{p:unique.idem}
Let $s\in S$. Then $s^+$ is the unique idempotent in $\til H_s$.
\end{Prop}
\begin{proof}
If $e\in \til H_s$ is an idempotent, then $e=e^+=s^+$, as required.
\end{proof}

 In order to imitate our approach to inverse semigroup algebras~\cite{mobius2} in this setting, we need to define a partial order on $S$.

\begin{Prop}\label{p:po}
 Let $S$ be an idempotent finite semigroup with central idempotents.  Define $s\leq t$ if $s=te$ with $e\in E(S)$.
\begin{enumerate}
\item $\leq$ is a partial order extending the order on $E(S)$.
\item $s\leq t$ and $s'\leq t'$ implies $ss'\leq tt'$.
\item $s\leq t$ if and only if $s=ts^+$.
\item If $a\leq st$, then there exist unique $s_0\leq s$ and $t_0\leq t$ with $a=s_0t_0$ and $s_0^+=a^+=t_0^+$.
\end{enumerate}
\end{Prop}
\begin{proof}
Since $s=ss^+$, the relation $\leq$ is reflexive.  It is also transitive as $s\leq t$ and $t\leq u$, implies that $s=te$ and $t=uf$ with $e,f\in E(S)$.  But then $s=ufe$ and $fe\in E(S)$, and so $s\leq u$.  Suppose that $s\leq t$ and $t\leq s$.  Then $s=te$ and $t=sf$ with $e,f\in E(S)$.  But then  $s=te=sfe=sef=teef=tef=sf=t$.  Therefore, $\leq$ is a partial order and it clearly  extends the order on $E(S)$.  If $s\leq s'$ and $t\leq t'$, we can write $s=s'e$ and $t=t'f$ with $e,f\in E(S)$.  Then $s't'ef=s'et'f=st$ and so $st\leq s't'$ as $ef\in E(S)$.  If $s=ts^+$, then obviously $s\leq t$. If $s\leq t$, then $s=te$ with $e\in E(S)$.  But then $se=s$ and so $s^+\leq e$ in $E(S)$. Therefore, $s=ss^+=tes^+=ts^+$, as required.

For the final property, put $s_0 = sa^+$ and $t_0 = ta^+$.  Then  $s_0\leq s$, $t_0\leq t$ by definition and $s_0t_0 = sta^+=a$ by~(3).  By construction, $s_0a^+=s_0$.  If $e\in E(S)$ with $s_0e=s_0$, then $ae= s_0t_0e=s_0et_0=s_0t_0 =a$ and so $a^+\leq e$.  Thus $s_0^+=a^+$ and, similarly, $t_0^+=a^+$.  The uniqueness is clear since if $s_1\leq s$ and $s_1^+=a^+$, then $s_1=sa^+$ by (3) and similarly for $t$.
\end{proof}

Observe that if $e\in E(S)$, then $s\in eSe$ if and only if $se=s$, if and only if, $s^+\leq e$.

\begin{Prop}\label{p:id.quotients}
Let $e\in E(S)$ and $I_e=\{s\mid s^+<e\}$.  Then $I_e$ is an ideal of $S$.
\end{Prop}
\begin{proof}
Let $s\in I_e$ and $t\in S$. Then $(st)s^+=ss^+t=st$ and so $(st)^+\leq s^+<e$.  Similarly, $(ts)s^+ = ts$ and so $(ts)^+<e$.
\end{proof}

Recall that the maximal subgroup $G_e$ of $S$ at $e$ is the group of units of $eSe$.
Note that $I_e=\emptyset$ if and only if $eSe=G_e=\til H_e$.  Indeed, $eSe\setminus I_e=\til H_e$ and so if $I_e=\emptyset$, then $\til H_e=eSe$ is a finite monoid with unique idempotent $e$.  But that makes $eSe$ a group and so $G_e=eSe$.
For $e\in E(S)$, we put $\til H_e^0 = eSe/I_e$, the Rees quotient, if $I_e\neq \emptyset$ and if $I_e=\emptyset$, then we put $\til H_e^0=G_e\cup \{z\}$ where $z$ is an adjoined zero (recall that $\til H_e=G_e$ in this case).  Note that $\til H_e^0$ is a monoid with identity $e$ and can be identified with $\til H_e\cup \{z\}$ where $z$ is a zero and, for $a,b\in \til H_e$, we have that
\[a\cdot b = \begin{cases} ab, & \text{if}\ ab\in \til H_e\\ z, & \text{else.}\end{cases}\]  From now on we just write $ab$ as no confusion should arise.

 The next proposition says that $\til H_e^0$ has a structure similar to that of a local ring.

\begin{Prop}\label{p:is.local}
If $e\in E(S)$, then the group of units of $\til H_e^0$ is $G_e$ and every nonunit is nilpotent.
\end{Prop}
\begin{proof}
Since each element of $G_e$ generates the same ideal as $e$ and $e\notin I_e$, we deduce that $G_e\subseteq eSe\setminus I_e=\til H_e$.  Since an element of $\til H_e$ is a unit in $\til H_e^0$ if and only if it is in $eSe$, the first claim follows.  Suppose that $s\in \til H_e$. Then $s^n$ is an idempotent $f$ for some $n>0$. If $f=e$, then $s\in G_e$ (with inverse $s^{n-1}$).  If $f\neq e$, then $f\notin \til H_e$ by Proposition~\ref{p:unique.idem}, i.e., $f\in I_e$.  Thus $s^n=z$ in $\til H_e^0$.
\end{proof}

Monoids in which every nonunit is nilpotent are called elementary monoids in~\cite{PoniFrob} and~\cite[Chapter~16]{oknisemigroupalgebra}. They are precisely the finite monoids in which the identity and zero are the only idempotents.

We can now prove an analogue of Theorem~\ref{t:algebra.isom} for idempotent semigroups with central idempotents.    Note that the algebra  $A=\prod_{e\in E(S)} K_0\til H_e^0$, viewed as an internal direct sum,  is a $K$-vector space with basis the disjoint union of the $\til H_e$, which in turn is just $S$ by Proposition~\ref{p:unique.idem}.  From this point of view, the product of two elements $s,t\in S$ is nonzero in $A$ if and only if $s^+=t^+=(st)^+$, in which case their product in $A$ is their product in $S$.

\begin{Thm}\label{t:central.idems}
Let $S$ be an idempotent finite semigroup with central idempotents.  Then the mapping $Z\colon KS\to \prod_{e\in E(S)} K_0\til H_e^0$ given by $Z(s) = \sum_{t\leq s} t$ on $s\in S$ is an isomorphism of $K$-algebras.  In particular, $KS$ is unital.
\end{Thm}
\begin{proof}
Since the matrix of $Z$ with respect to the natural basis $S$ for both algebras is the zeta-function of $S$, it follows that $Z$ is invertible with inverse \[s\mapsto \sum_{t\leq s}\mu_S(t,s)t.\]  It remains to check that $Z$ is a homomorphism.  This is immediate from 
Proposition~\ref{p:po}(2),(4):
\[Z(st) = \sum_{a\leq st} a = \sum_{\substack{s_0\leq s,t_0\leq t\\ s_0^+=t_0^+=(s_0t_0)^+}} s_0t_0 = \sum_{s_0\leq s}s_0\cdot \sum_{t_0\leq t} t_0 = Z(s)Z(t),\] as required (where the products in the two rightmost terms are taken in $\prod_{e\in E(S)} K_0\til H_e^0$).
\end{proof}

Theorem~\ref{t:central.idems} generalizes a theorem of Ponizovski\u{\i}~\cite{PoniFrob}, which proves that if $S$ is an idempotent finite commutative semigroup, then $KS$ is isomorphic to a direct product of contracted semigroup algebras of monoids with no idempotents except $0$ and $1$.  However, our isomorphism is explicit and, in particular, we can write down the inverse isomorphism using the M\"obius function of $S$, which is necessary in order to compute the semigroup determinant; for example, the proof of Ponizovski\u{\i} would not be explicit enough to recover the Wilf-Lindstr\"om theorem for the semigroup determinant of a meet semilattice.

The general machinery now computes the semigroup determinant of $S$ in terms of the contracted semigroup determinants of the $\til H_e^0$.

\begin{Thm}\label{t:ze.case}
Let $S$ be a finite idempotent semigroup with central idempotents. For $s\in S$, put $y_s = \sum_{t\leq s}\mu_S(t,s)x_t$.  Then 
\[\theta_S = \prod_{e\in E(S)}\til\theta_{\til H_e^0}(Y_e)\]  where  $Y_e = \{y_s\mid s\in \til H_e\}$.
\end{Thm}
\begin{proof}
The matrix of the isomorphism $Z$ from Theorem~\ref{t:central.idems} is the
zeta-function of $S$.  Therefore, $\til Z\inv (x_s) = \sum_{t\leq s} \mu_S(t,s)x_t=y_s$.  Since $Z$ is unimodular, Proposition~\ref{p:direct.prod} and Theorem~\ref{t:cvt}  give the desired result.
\end{proof}

Note that if $S$ is an inverse semigroup with central idempotents, then $\til H_e=G_e$ for each idempotent $e$ and Theorem~\ref{t:ze.case} reduces to Theorem~\ref{t:comm.case}.

\subsection{Commutative semigroups}
By Corollary~\ref{c:squared.algebra}, the computation of the semigroup determinant of a commutative semigroup $S$ is reduced to the case that $S^2=S$.
Since commutative semigroups obviously have central idempotents,  the computation of $\theta_S$ is reduced  by Theorem~\ref{t:ze.case} to the case of the contracted semigroup determinant of a commutative monoid $M$ with zero $z$ and group of units $G$ such that each nonunit is nilpotent.  These were studied by Ponizovski\u{\i}~\cite{PoniFrob} and Wenger~\cite{Wenger} in an attempt to understand when the algebra of a finite commutative semigroup is Frobenius over an arbitrary field.  Here we stick with the field of complex numbers, which is a simpler case, but we need to give more precise results in order to compute explicitly the semigroup determinant.

If $M$ is a commutative monoid with group of units $G$, then $M/G =\{Gm\mid m\in M\}$ is a monoid with multiplication $Gm_1\cdot Gm_2=Gm_1m_2$, and is a quotient of $M$ via $m\mapsto Gm$.

Let $\wh{G}$ denote the group of characters $\chi\colon G\to \mathbb C^\times$. Note that if $|G|=n$, then $\chi(G)\leq \mu_n$ where $\mu_n$ is the group of $n^{th}$-roots of unity.  Thus $\chi(g\inv)=\overline{\chi(g)}$ for all $g\in G$.  Recall that if \[e_{\chi} = \frac{1}{|G|}\sum_{g\in G}\ov{\chi(g)}g,\] then $1=\sum_{\chi\in \wh G}e_{\chi}$ is a decomposition into orthogonal primitive idempotents of $\mathbb CG$ and $e_{\chi}\mathbb C_0M$ is the isotypic component of $\chi$ for the left $\mathbb CG$-module structure on $\mathbb C_0M$. We remark that $ge_{\chi}=\chi(g)e_{\chi}=e_{\chi}g$ for all $g\in G$.

Note that since $\mathbb CG\leq \mathbb C_0M$ and $\mathbb C_0M$ is commutative, we have an isomorphism of $\mathbb C$-algebras,
\begin{equation}\label{eq:local.decomp.artin}
\mathbb C_0M\cong \prod_{\chi\in \wh G} e_{\chi}\mathbb C_0Me_{\chi}=\prod_{\chi\in \wh G}e_{\chi}\mathbb C_0M.
\end{equation}
  We will show that each $e_{\chi}\mathbb C_0M$ is a twisted contracted monoid algebra of a commutative monoid whose nonidentity elements are all nilpotent;  if $M/G$ admits a transversal which is a submonoid, then all these algebras are untwisted.  Note that such twisted contracted monoid algebras are local by the discussion in Section~\ref{s:nilpotents} and so the $e_{\chi}$ are still primitive in $\mathbb C_0M$ and  \eqref{eq:local.decomp.artin} is the canonical decomposition of the finite dimensional commutative algebra $\mathbb C_0M$ into a product of local rings.

\begin{Prop}\label{p:bass.analogy}
Let $M$ be a commutative monoid with zero $z$ and group of units $G$ such that $M\setminus G$ consists of nilpotent elements.  Then $Mm=Mm'$ if and only if $Gm=Gm'$.
\end{Prop}
\begin{proof}
The only non-trivial implication is showing that $Mm=Mm'$ implies $Gm=Gm'$.  Assume that $Mm=Mm'$ is the case.  Then $xm=m'$ and $ym'=m$ for some $x,y\in M$.  If $x$ and $y$ are units, there is nothing to prove.  Suppose that without loss of generality that $x$ is not a unit.  Then $x$ and hence $xy$ is nilpotent.  But $xym'=xm=m'$ and hence by nilpotence of $xy$, $z=zm'=(xy)^nm'=m'$ for $n$ large enough.  Similarly, $xym=yxm=ym'=m$ and so $m=z$.  Thus $Gm=Gm'$ as $m=z=m'$.
\end{proof}

 Fix once and for all representatives $m_1,\ldots, m_r$ of the orbits of $G$ on $M\setminus \{z\}$ and without loss of generality assume that $m_1=1$.
  If $m_im_j\neq z$, then define $f(i,j)$ by $m_im_j\in Gm_{f(i,j)}$.
  Let $G_i$ be the stabilizer in $G$  of $m_i$; since $G$ is commutative, $G_i$ is the pointwise stabilizer of the orbit $Gm_i$.

\begin{Prop}\label{p:char.ideal}
Let $\chi\in \wh G$ and $I_{\chi}$ consist of those elements of $m\in M$ whose stabilizers in $G$ are not contained in $\ker \chi$.  Then $I\chi$ is empty if $\chi$ is the trivial character $1_G$ and otherwise is a proper ideal of $M$.
\end{Prop}
\begin{proof}
Since $\ker 1_G=G$, trivially $I_{1_G}=\emptyset$.  Since the stabilizer of $z$ is $G$, if $\chi$ is a non-trivial character, then $z\in I_{\chi}$.  Moreover, if $m\in I_{\chi}$ and $gm=m$ with $g\notin\ker \chi$, then for any $m'\in M$, we have $g(mm')=(gm)m'=mm'$ and so $g$ stabilizes $mm'$.  Thus $mm'\in I_{\chi}$ and so $I_{\chi}$ is an ideal.  It is proper, because the stabilizer of $1$ is trivial and hence $1\notin I_{\chi}$.
\end{proof}

For $\chi\in \wh G$, let us put
\[M_{\chi} = \begin{cases}M, & \text{if}\ \chi=1_G\\ M/I_{\chi},&  \text{else.}\end{cases}\]
Our goal is to compute the contracted semigroup determinant of $M$ in terms of paratrophic determinants of twisted contracted monoid algebras of the $M_{\chi}/G$.   Note that in $M_{\chi}/G$ every nonidentity element is nilpotent, i.e., $M_{\chi}/G$ is the result of adjoining an identity to a nilpotent commutative semigroup and so we are in the situation dealt with in Section~\ref{s:nilpotents}.

If $m\in M\setminus I_{\chi}$, we can put $\chi(m) = \chi(g)$ where $m=gm_i$ with $g\in G$.  This is well defined because if $m=hm_i$, then $h\inv g\in G_i\subseteq \ker \chi$. Note that this definition does depend on the choice of $m_i$, but we have fixed $m_1,\ldots, m_r$ once and for all. Since $m_1=1$, this definition of $\chi$ agrees with the original definition on $G$.  Let us fix the notation $J_{\chi} = \{i\mid G_i\subseteq \ker\chi\}$.

The first part of the following proposition is essentially Wood's ``good basis'' from~\cite{woodsemigroup}, generalized to our setting.

\begin{Prop}\label{p:good.basis}
Let $\chi\in \wh G$. Then the following hold.
\begin{enumerate}
\item A basis $B_{\chi}$ for $e_{\chi}\mathbb C_0M = e_{\chi}\mathbb C_0Me_{\chi}$ consists of the elements \[m_{i,\chi}=e_{\chi}m_i =\frac{1}{|Gm_i|}\sum_{m\in Gm_i}\overline{\chi(m})m\] with $i\in J_{\chi}$.
\item $e_{\chi}\mathbb C_0M\cong \mathbb C_0(M_{\chi}/G,c_{\chi})$ where $c_{\chi}(Gm_i,Gm_j)=\chi(m_im_j)$  if $m_im_j\notin I_{\chi}$.  Furthermore, if $m_im_j = m_{f(i,j)}$ whenever $m_im_j\notin I_{\chi}$, then $e_{\chi}\mathbb C_0M\cong \mathbb C_0[M_{\chi}/G]$.
\item $\til \theta_{M_{\chi}/G,c_{\chi}}=0$ unless $M_{\chi}/G$ has a unique annihilating element $Gm_{i_{\chi}}$, in which case $\til\theta_{M_{\chi}/G,c_{\chi}} = \det A(\chi)\cdot x_{Gm_{i_{\chi}}}^{|M_{\chi}/G|-1}$ where $A(\chi)$ is the $J_{\chi}\times J_{\chi}$-matrix with
\[A(\chi)_{ij} = \begin{cases} \chi(m_im_j), & \text{if}\ m_im_j\in Gm_{i_{\chi}}\\ 0, & \text{else}\end{cases}\] for $i,j\in J_{\chi}$.
\end{enumerate}
\end{Prop}
\begin{proof}
We have  that $\mathbb C_0M= \bigoplus_{i=1}^r \mathbb CGm_i$ as a left $\mathbb CG$-module.  Moreover, $\mathbb CGm_i\cong \mathbb C[G/G_i]$ and the latter is induced from the trivial representation of $G_i$.  Therefore, by Frobenius reciprocity, the multiplicity of $\chi\in \wh G$ in $\mathbb CGm_i$ is $1$ if $G_i\subseteq \ker \chi$, i.e., $i\in J_\chi$, and $0$,  otherwise.  Since $e_{\chi}\mathbb CGm_i = \mathbb CGe_{\chi}m_i$, we deduce that $e_{\chi}m_i$ is a basis for the one-dimensional isotypic component of $\chi$ in $\mathbb CGm_i$ when $i\in J_{\chi}$, and only such $i$ contribute to the isotypic component $e_{\chi}\mathbb C_0M$.  Thus $B_{\chi}$ is a basis for $e_{\chi}\mathbb C_0M$.  Let $T$ be a set of coset representatives for $G/G_i$.   We next compute, using $G_i\subseteq\ker \chi$ that
\begin{align*}
e_{\chi}m_i = \frac{1}{|G|}\sum_{g\in G}\overline{\chi(g)}gm_i &= \frac{1}{|G|}\sum_{h\in T}\overline{\chi(h)}h\sum_{u\in G_i}\overline{\chi(u)}um_i\\ & = \frac{|G_i|}{|G|}\sum_{h\in T}\overline{\chi(h)}hm_i \\ &= \frac{1}{|Gm_i|}\sum_{m\in Gm_i}\overline{\chi(m)}m.
\end{align*}

We next check that $e_{\chi}\mathbb C_0M\cong \mathbb C_0(M_{\chi}/G,c_{\chi})$  by showing that $m_{i,\chi}m_{j,\chi}$ is zero if $m_im_j\in I_{\chi}$ and is otherwise $\chi(m_im_j)m_{f(i,j),\chi}$.  The desired isomorphism will follow.   Clearly, $e_{\chi}m_ie_{\chi}m_j = e_{\chi}m_im_j$ and this will be zero unless $m_im_j\notin I_{\chi}$ by our previous discussion.  If $m_im_j\notin I_{\chi}$, then $m_im_j = hm_{f(i,j)}$ with $h\in G$, and so $e_{\chi}m_im_j =e_{\chi} hm_{f(i,j)} = \chi(h)e_{\chi}m_{f(i,j)} = \chi(m_im_j)m_{f(i,j),\chi}$, as required.  In the case that $m_im_j = m_{f(i,j)}$, then $h\in G_{f(i,j)}\subseteq \ker \chi$ and so $\chi(h)=1$.

The third item follows from Theorem~\ref{t:nilvanish} since $M_{\chi}/G$ is a nilpotent semigroup with adjoined identity.
This completes the proof.
\end{proof}

The special case of this theorem where there is a set of representatives for $M/G$ forming a submonoid, and so all the twists $c_{\chi}$ are trivial, is the content of~\cite[Lemma~7]{Wenger}.  Our result clarifies the general case.

Note that $\chi$ and $\overline{\chi}$ have the same kernel and hence $G_i\subseteq \ker \chi$ if and only if $G_i\subseteq \ker \ov{\chi}$, a fact that we shall exploit without comment.

\begin{Lemma}\label{l:messy.det}
Let $f\colon \mathbb C_0M\to \prod_{\chi\in \wh G} \mathbb C_0(M_{\chi},c_{\chi})$ be the isomorphism coming from Proposition~\ref{p:good.basis} and \eqref{eq:local.decomp.artin}, and let $P$ be the matrix of $f$ with respect to the bases $M\setminus \{z\}$ and the union of the $B_{\chi}$ from Proposition~\ref{p:good.basis} with $\chi\in \wh G$ (identifying $e_{\chi}\mathbb C_0M$ with $\mathbb C_0(M_{\chi},c_{\chi})$).  Then $\det (PP^T) = \pm\prod_{\chi\in \wh G}\prod_{i\in J_{\chi}}|Gm_i|$.
\end{Lemma}
\begin{proof}
Let $m\in Gm_i$.  If $i\in J_{\chi}$ and $m=gm_i$, then $e_{\chi}m = e_{\chi}gm_i = \chi(g)e_{\chi}m_i = \chi(m)m_{i,\chi}$.  Otherwise, $e_{\chi}m=0$.  Therefore, \[P_{m_{i,\chi},m} = \begin{cases}\chi(m), & \text{if}\ m\notin I_{\chi}\\ 0, & \text{else}\end{cases}\] and $P_{m_{j,\lambda},m}=0$ if $j\neq i$.   It follows that $PP^T_{m_{i,\chi}, m_{j,\lambda}}$ is zero unless $i=j$ and $G_i\subseteq \ker \chi\cap \ker\lambda$, in which case it is
\begin{equation}\label{eq:use.Fourier}
\sum_{m\in Gm_i} \chi(m)\lambda(m) = \frac{1}{|G_i|}\sum_{g\in G}\chi(g)\lambda(g)
\end{equation}
 (since each element $m\in Gm_i$ can be written as $gm_i$ in exactly $|G_i|$ ways and $\chi(m)=\chi(g)$, $\lambda(m)=\lambda(g)$ whenever $m=gm_i$).  But by the orthogonality relations for characters, the right hand side of \eqref{eq:use.Fourier} is zero unless $\chi =\overline{\lambda}$, in which case it is $|G|/|G_i|= |Gm_i|$.  If we permute the rows by switching $m_{i,\chi}$ and $m_{i,\ov\chi}$ for $i\in J_{\chi}$, we turn $PP^T$ into a diagonal matrix where the $m_{i,\chi}$ diagonal entry is $|Gm_i|$.  The result then follows.
\end{proof}

\begin{Rmk}\label{r:sign}
The sign in the above lemma can be computed as follows.  Let $D$ be a set of representatives of the complex conjugate pairs of non-real-valued characters of $G$.  Then the sign in $\det (PP^T)$ is $\prod_{\chi\in D}(-1)^{|J_{\chi}|}$.
\end{Rmk}

\begin{Thm}\label{t:local.like.case}
Let $M$ be a commutative monoid with zero and group of units $G$ such that $M\setminus G$ consists of nilpotent elements.  Then $\til\theta_M=0$ unless each $M_{\chi}/G$ has a unique annihilating element $Gm_{i_{\chi}}$, in which case \[\til\theta_M = \pm\prod_{\chi\in \wh G}\left(\prod_{i\in J_{\chi}}\frac{|Gm_i|}{|G|}\right)\det A(\chi)\cdot \left(\sum_{g\in G}\ov{\chi(g)}x_{gm_{i_{\chi}}}\right)^{|M_{\chi}/G|-1}\] where we have retained the previous notation (including that of Proposition~\ref{p:good.basis}).
\end{Thm}
\begin{proof}
This follows from Proposition~\ref{p:good.basis}, Proposition~\ref{p:direct.prod}, Theorem~\ref{t:cvt} (because $\frac{1}{|G|}\sum_{g\in G}\overline{\chi(g)}gm_{i_{\chi}}$ is the preimage of $m_{i_{\chi},\chi}$ under the isomorphism in Proposition~\ref{p:good.basis}), Lemma~\ref{l:messy.det} and since $|J_{\chi}| = |M_{\chi}/G|-1$.
\end{proof}

Let us give two illustrative examples of Theorem~\ref{t:local.like.case}.  Let $M$ be the commutative monoid with zero $z$ with generators $a,r,s,t$ and  relations $a^2=1$, $rt=s^2=z$, $r^2=rs=ast=at^2$ and  $w(r,s,t)=z$ whenever $w$ is a word of length at least $3$. Put $z'=r^2$. Then \[M=\{1,a,r,ar,s,as,t,at,z',az',z\}\] is an $11$-element monoid with group of units $G=\{1,a\}$ and $M\setminus G$ is $3$-nilpotent.  Moreover, the stabilizer in $G$ of every nonzero element is trivial.   Then $Gz'$ is the unique annihilating element of $M/G$.  Note that $M/G$ is isomorphic to the monoid in Proposition~\ref{p:3nil} with \[B=\begin{bmatrix} 1 & 1& 0 \\ 1 &0&1\\ 0&1&1\end{bmatrix}.\]  In particular, $\det B=-2$ and so $\mathbb C_0[M/G]$ is Frobenius by Proposition~\ref{p:3nil}.  But we shall see that $\mathbb C_0M$ is not.
Let $1_G$ be the trivial character of $G$ and $\chi$ the non-trivial character.  Note that $M_{1_G}=M=M_{\chi}$, as all nonzero elements have trivial stabilizers.  Let us take $1,r,s,t,z'$ as our set of orbit representatives for $(M\setminus \{z\})/G$. Note that $\chi(z')=1$ and $\chi(az')=-1$.    Therefore,
\[A(1_G) = \begin{bmatrix} 0& 0& 0& 0 & 1\\ 0& 1 & 1& 0&0  \\ 0&1 &0&1&0\\ 0&0&1&1&0\\ 1& 0&0&0&0\end{bmatrix},\qquad   A(\chi) = \begin{bmatrix} 0&0&0&0&1\\ 0& 1 & 1& 0 &0 \\ 0& 1 &0&-1&0\\ 0& 0&-1&-1&0\\ 1&0&0&0&0\end{bmatrix}.\]  In particular, $\det A(\chi)=0$ and so $\til\theta_M=0$ by Theorem~\ref{t:local.like.case}.  Therefore, $\mathbb C_0M$ is not Frobenius.  Incidentally, this gives an example where the converse
of~\cite[Chapter~16, Theorem~25]{oknisemigroupalgebra} fails over the complex numbers.

Our next example is from Wenger~\cite{Wenger} and was given as an example where~\cite[Lemma~7]{Wenger} does not apply, but the algebra is still Frobenius.  He showed the algebra is Frobenius by an ad hoc substitution of the variables in the contracted semigroup determinant. Here, we use Theorem~\ref{t:local.like.case} to give the exact factorization of the contracted semigroup determinant.  Let $M$ be the commutative monoid with zero $z$ with generators $a,r,s$ and defining relations $a^2=1$, $rs=r^3=s^3=z$ and $r^2=as^2$.  It is easy to see that $G=\{1,a\}$ is the group of units of $M$, that $M\setminus G$ is $3$-nilpotent and that if we put $z'=r^2$, then $M = \{1,a,r,ar,s,as,z',az',z\}$. Moreover, $1,r,s,z'$ form a set of representatives for $(M\setminus \{z\})/G$ and $Gz'$ is the unique annihilating element of $M/G$.  The stabilizer in $G$ of each nonzero element of $M$ is trivial.  Note that $r^2=(ar)^2=z'$ and $s^2=(as)^2=az'$, and so there is no way to choose a set of orbit representatives that form a submonoid.  Let $1_G$ be the trivial character of $G$ and $\chi$ the nontrivial character.  Again, we have $\chi(z')=1$ and $\chi(az')=-1$.  We then have
\[A(1_G) = \begin{bmatrix}  0& 0& 0 & 1\\ 0& 1 & 0& 0  \\ 0& 0&1&0\\ 1 & 0&0&0\end{bmatrix},\qquad   A(\chi) = \begin{bmatrix}  0& 0& 0 & 1\\ 0& 1 & 0& 0  \\ 0& 0&-1&0\\ 1 & 0&0&0\end{bmatrix}\] and so $\det A(1_G) = -1$, $\det A(\chi) = 1$, which implies that $\mathbb C_0M$ is Frobenius.   Theorem~\ref{t:local.like.case} shows that $\til \theta_M = \pm (x_{z'}+x_{az'})^4(x_{z'}-x_{az'})^4$ as each orbit of $G$ on $M\setminus \{z\}$ has size $2$ (actually the sign is negative by Remark~\ref{r:sign} since all the characters of $G$ are real-valued).

We now specialize Theorem~\ref{t:local.like.case} to the case where the ideal of nonunits of $M$ is principal.  This includes the multiplicative monoid of a chain ring and so our result generalizes that of Wood~\cite{woodsemigroup}.

\begin{Cor}\label{c:woodlike}
Let $M$ be a commutative monoid with zero $z$  and group of units $G$ such that the ideal $M\setminus G$  of nonunits is principal and generated by a nilpotent element $t$.  Let $G_i$ be the stabilizer of $t^i$ in $G$ for $0\leq i\leq k$ where $k\geq 0$ is largest with $t^k\neq z$.    For $\chi\in \wh G$, let $i_{\chi}$ be the largest index $i$ with $G_i\subseteq \ker \chi$.
Then the ideals of $M$ are the $Mt^i$ with $0\leq i\leq k+1$, which form a chain, and
\[\til\theta_M = \pm\prod_{\chi\in \wh G}\left(\prod_{0\leq i\leq i_{\chi}}\dfrac{1}{|G_i|}\right)\left(\sum_{g\in G}\overline {\chi(g)}x_{gt^{i_{\chi}}}\right)^{i_{\chi}+1}.\]
\end{Cor}
\begin{proof}
Notice that $M\setminus G$ consists of nilpotent elements because $t$ is nilpotent. Also, since $t^{k+1}=z$ and $t^k\neq z$, we have $Mt^{i}\neq Mt^{i+1}$ for $0\leq i\leq k$.
First observe that each element  $m\in M\setminus \{z\}$ can be written in the form $m=gt^i$ with $g\in G$ and $0\leq i\leq k$. The claim is trivial if $m\in G$.  Else, since $M\setminus G = Mt$ and $m\neq z$, there exists $1\leq i\leq k$ with $m\in Mt^i\setminus Mt^{i+1}$.  Write $m=at^i$ with $a\in M$.  If $a\in G$, we are done.  Else $a\in M\setminus G=Mt$ and so $a=bt$ with $b\in M$.  But then $m=bt^{i+1}$, a contradiction.    It follows that $1,t,\ldots, t^k,z$ is a transversal to $M/G$ which is a submonoid and that the $Mt^i$ with $0\leq i\leq k+1$ are all the ideals of $M$.   Clearly, $G_0\subseteq G_1\subseteq\cdots \subseteq G_k$.   Also, $M/G$ is  isomorphic to the cyclic monoid $N=\langle t\rangle = \{1, t,\ldots, t^k,t^{k+1}=z\}$.  Note that if $\chi\in \wh G$, then $M_{\chi} = M/Mt^{i_{\chi}+1}$ and hence $M_{\chi}/G$ is a nilpotent cyclic semigroup of nilpotency index $i_{\chi}+1$ with an adjoined identity.  Moreover, $e_{\chi}\mathbb CM$ is just the contracted semigroup algebra of $M_{\chi}/G$ by Proposition~\ref{p:good.basis}, since our transversal to the orbits is a submonoid.  The result now follows from Theorem~\ref{t:local.like.case}, Proposition~\ref{p:cyclic.nilpotent} and the observation that $|Gt^i|/|G| = 1/|G_i|$.
\end{proof}

Notice that  if the ideals of a monoid form a chain, then each ideal is clearly principal.  So the hypotheses of Corollary~\ref{c:woodlike} could alternatively be stated as $M$ is a commutative monoid with zero such that all nonunits are nilpotent and the ideals of $M$ form a chain.

We now arrive at the following computation of the Dedekind-Frobenius determinant of a finite commutative semigroup, generalizing Dedekind's theorem for abelian groups, the Wilf-Lindstr\"om theorem for semilattices, Wood's theorem for commutative chain rings  and Theorem~\ref{t:comm.case}.  In order to state our theorem, we unfortunately will need to introduce some cumbersome notation.  Suppose that $S$ is a finite commutative semigroup with $S^2=S$. For each idempotent $e$, we fix a set $T_e$ of representatives of the orbits of $G_e$ on $\til H_e$ containing $e$.  If $\chi\in \wh G_e$, then $T_{e,\chi}$ will denote the set of elements of $T_e$ whose stabilizer in $G_e$ is contained in $\ker \chi$.  If $(\til H_e^0/I_{\chi})/G_e$ has a unique annihilating element, we denote the representative of that orbit in $T_{e,\chi}$ by $s_{\chi}$.  If this is the case, we define $A(\chi)$ to be the $T_{e,\chi}\times T_{e,\chi}$-matrix with
\[A(\chi)_{s,t} = \begin{cases}\chi(g), & \text{if}\ st=gs_{\chi}\\ 0, & \text{else.}\end{cases}\] This is well defined, i.e., doesn't depend on the choice of $g$.

\begin{Thm}\label{t:comm.case.final}
Let $S$ be a finite commutative semigroup.  If $S^2\neq S$, then $\theta_S=0$. If $S^2=S$, then we retain the notation $T_{e,\chi}$, $A(\chi)$ and $s_{\chi}$ from above for $\chi\in \wh G_e$ with $e\in E(S)$.
In order for $\theta_S$ not to be identically zero, for each $e\in E(S)$ and each $\chi\in \wh{G_e}$, the element $s_{\chi}$ should exist, i.e., $(\til H_e^0/I_{\chi})/G_e$ should have a unique annihilating element.  
Then we have
\begin{align*}
\theta_S =\pm\prod_{e\in E(S)}\prod_{\chi\in \wh {G_e}}\left(\prod_{s\in T_{e,\chi}}\frac{|G_es|}{|G_e|}\right)& \det A(\chi) \\ & \cdot \left(\sum_{g\in G_e}\ov{\chi(g)}\sum_{t\leq gs_\chi}\mu_S(t,gs_{\chi})x_t\right)^{|T_{e,\chi}|}
\end{align*}
 with $\mu_S$ the M\"obius function of $S$ with respect to the partial order $s\leq t$ if $s\in tE(S)$.
\end{Thm}
\begin{proof}
According to Corollary~\ref{c:squared.algebra}, if $S^2\neq S$, then $\theta_S=0$.  Assuming $S^2=S$, the result follows by putting together Theorems~\ref{t:ze.case} and~\ref{t:local.like.case}.
\end{proof}

Theorem~\ref{t:comm.case.final} shows that the semigroup determinant of a finite commutative semigroup $S$ is either identically zero, or it factors into linear factors which are determined by the characters of the maximal subgroups of $S$ (or, equivalently, the characters of $S$) and the M\"obius function of $S$.  These factors appear in general with multiplicities.  Note that the sign appearing in Theorem~\ref{t:comm.case.final} is easy to compute using Remark~\ref{r:sign}.  Also observe that $\prod_{\chi\in \wh{G_e}}\det A(\chi)$ is an integer since the entries of $A(\chi)$ are $|G_e|^{th}$-roots of unity, and the Galois group of the corresponding cyclotomic field extension permutes the matrices $A(\chi)$ according to the action of the Galois group on $\wh G_e$ and hence fixes $\prod_{\chi\in \wh{G_e}}\det A(\chi)$, which is then a rational algebraic integer.

\begin{Rmk}
Note that Theorem~\ref{t:central.idems}, Proposition~\ref{p:good.basis} and  \eqref{eq:local.decomp.artin} imply that if $S$ is a finite commutative semigroup with $S^2=S$, then $\mathbb CS$ is isomorphic to a finite direct product of local rings that are twisted contracted monoid algebras of commutative nilpotent semigroups with adjoined identity.
\end{Rmk}

\subsection*{Acknowledgments}
The author thanks Jay Wood for a number of helpful discussions.


\begin{thebibliography}{10}

\bibitem{tha}
M.~Aguiar and S.~Mahajan.
\newblock {\em Topics in hyperplane arrangements}, volume 226 of {\em
  Mathematical Surveys and Monographs}.
\newblock American Mathematical Society, Providence, RI, 2017.

\bibitem{benson}
D.~J. Benson.
\newblock {\em Representations and cohomology. {I}}, volume~30 of {\em
  Cambridge Studies in Advanced Mathematics}.
\newblock Cambridge University Press, Cambridge, second edition, 1998.
\newblock Basic representation theory of finite groups and associative
  algebras.

\bibitem{BHR}
P.~Bidigare, P.~Hanlon, and D.~Rockmore.
\newblock A combinatorial description of the spectrum for the {T}setlin library
  and its generalization to hyperplane arrangements.
\newblock {\em Duke Math. J.}, 99(1):135--174, 1999.

\bibitem{Brown1}
K.~S. Brown.
\newblock Semigroups, rings, and {M}arkov chains.
\newblock {\em J. Theoret. Probab.}, 13(3):871--938, 2000.

\bibitem{Brown2}
K.~S. Brown.
\newblock Semigroup and ring theoretical methods in probability.
\newblock In {\em Representations of finite dimensional algebras and related
  topics in Lie theory and geometry}, volume~40 of {\em Fields Inst. Commun.},
  pages 3--26. Amer. Math. Soc., Providence, RI, 2004.

\bibitem{ClarkTwist}
W.~E. Clark.
\newblock Twisted matrix units semigroup algebras.
\newblock {\em Duke Math. J.}, 34:417--423, 1967.

\bibitem{Clifford}
A.~H. Clifford.
\newblock Semigroups admitting relative inverses.
\newblock {\em Ann. of Math. (2)}, 42:1037--1049, 1941.

\bibitem{CP}
A.~H. Clifford and G.~B. Preston.
\newblock {\em The algebraic theory of semigroups. {V}ol. {I}}.
\newblock Mathematical Surveys, No. 7. American Mathematical Society,
  Providence, R.I., 1961.

\bibitem{ConradRep}
K.~Conrad.
\newblock The origin of representation theory.
\newblock {\em Enseign. Math. (2)}, 44(3-4):361--392, 1998.

\bibitem{Ltilde}
J.~Fountain, G.~M.~S. Gomes, and V.~Gould.
\newblock Enlargements, semiabundancy and unipotent monoids.
\newblock {\em Comm. Algebra}, 27(2):595--614, 1999.

\bibitem{Frobeniushimself}
F.~G. Frobenius.
\newblock Theorie der hyperkomplexen {G}r\"o{\ss}en {I}.
\newblock {\em Sitzungsberichte der Preussischen Akademie der Wissenschaften},
  pages 504--537, 1903.

\bibitem{Green}
J.~A. Green.
\newblock On the structure of semigroups.
\newblock {\em Ann. of Math. (2)}, 54:163--172, 1951.

\bibitem{LatinSquare}
K.~W. Johnson.
\newblock Latin square determinants.
\newblock In {\em Algebraic, extremal and metric combinatorics, 1986
  ({M}ontreal, {PQ}, 1986)}, volume 131 of {\em London Math. Soc. Lecture Note
  Ser.}, pages 146--154. Cambridge Univ. Press, Cambridge, 1988.

\bibitem{Kleitman}
D.~J. Kleitman, B.~R. Rothschild, and J.~H. Spencer.
\newblock The number of semigroups of order {$n$}.
\newblock {\em Proc. Amer. Math. Soc.}, 55(1):227--232, 1976.

\bibitem{Kovacs}
L.~G. Kov{\'a}cs.
\newblock Semigroup algebras of the full matrix semigroup over a finite field.
\newblock {\em Proc. Amer. Math. Soc.}, 116(4):911--919, 1992.

\bibitem{Mazorchuk}
G.~Kudryavtseva and V.~Mazorchuk.
\newblock On three approaches to conjugacy in semigroups.
\newblock {\em Semigroup Forum}, 78(1):14--20, 2009.

\bibitem{Lam2}
T.~Y. Lam.
\newblock {\em Lectures on modules and rings}, volume 189 of {\em Graduate
  Texts in Mathematics}.
\newblock Springer-Verlag, New York, 1999.

\bibitem{Lawson}
M.~V. Lawson.
\newblock {\em Inverse semigroups}.
\newblock World Scientific Publishing Co. Inc., River Edge, NJ, 1998.
\newblock The theory of partial symmetries.

\bibitem{latticedet2}
B.~Lindstr\"{o}m.
\newblock Determinants on semilattices.
\newblock {\em Proc. Amer. Math. Soc.}, 20:207--208, 1969.

\bibitem{Malandro2}
M.~E. Malandro.
\newblock Fast {F}ourier transforms for finite inverse semigroups.
\newblock {\em J. Algebra}, 324(2):282--312, 2010.

\bibitem{MSS}
S.~Margolis, F.~Saliola, and B.~Steinberg.
\newblock Combinatorial topology and the global dimension of algebras arising
  in combinatorics.
\newblock {\em J. Eur. Math. Soc. (JEMS)}, 17(12):3037--3080, 2015.

\bibitem{ourmemoirs}
S.~{Margolis}, F.~{Saliola}, and B.~{Steinberg}.
\newblock {Cell complexes, poset topology and the representation theory of
  algebras arising in algebraic combinatorics and discrete geometry}.
\newblock {\em Mem. Amer. Math. Soc.}, to appear.

\bibitem{McAlisterCharacter}
D.~B. McAlister.
\newblock Characters of finite semigroups.
\newblock {\em J. Algebra}, 22:183--200, 1972.

\bibitem{ringoids}
B.~Mitchell.
\newblock Rings with several objects.
\newblock {\em Advances in Math.}, 8:1--161, 1972.

\bibitem{Noether}
E.~Noether.
\newblock Hyperkomplexe {G}r\"{o}\ss en und {D}arstellungstheorie.
\newblock {\em Math. Z.}, 30(1):641--692, 1929.

\bibitem{oknisemigroupalgebra}
J.~Okni{\'n}ski.
\newblock {\em Semigroup algebras}, volume 138 of {\em Monographs and Textbooks
  in Pure and Applied Mathematics}.
\newblock Marcel Dekker Inc., New York, 1991.

\bibitem{putchasemisimple}
J.~Okni{\'n}ski and M.~S. Putcha.
\newblock Complex representations of matrix semigroups.
\newblock {\em Trans. Amer. Math. Soc.}, 323(2):563--581, 1991.

\bibitem{PoniFrob}
I.~S. Ponizovski\u{\i}.
\newblock The {F}robeniusness of the semigroup algebra of a finite commutative
  semigroup.
\newblock {\em Izv. Akad. Nauk SSSR Ser. Mat.}, 32:820--836, 1968.

\bibitem{qtheor}
J.~Rhodes and B.~Steinberg.
\newblock {\em The {$q$}-theory of finite semigroups}.
\newblock Springer Monographs in Mathematics. Springer, New York, 2009.

\bibitem{RhodesZalc}
J.~Rhodes and Y.~Zalcstein.
\newblock Elementary representation and character theory of finite semigroups
  and its application.
\newblock In {\em Monoids and semigroups with applications (Berkeley, CA,
  1989)}, pages 334--367. World Sci. Publ., River Edge, NJ, 1991.

\bibitem{serrerep}
J.-P. Serre.
\newblock {\em Linear representations of finite groups}.
\newblock Springer-Verlag, New York, 1977.
\newblock Translated from the second French edition by Leonard L. Scott,
  Graduate Texts in Mathematics, Vol. 42.

\bibitem{smith}
H.~J.~S. Smith.
\newblock On the value of a certain arithmetical determinant.
\newblock {\em Proc. London Math. Soc.}, (7):208--212, 1875/76.

\bibitem{Burnsidealgebra}
L.~Solomon.
\newblock The {B}urnside algebra of a finite group.
\newblock {\em J. Combinatorial Theory}, 2:603--615, 1967.

\bibitem{Solomonrook}
L.~Solomon.
\newblock Representations of the rook monoid.
\newblock {\em J. Algebra}, 256(2):309--342, 2002.

\bibitem{Stanley}
R.~P. Stanley.
\newblock {\em Enumerative combinatorics. {V}ol. 1}, volume~49 of {\em
  Cambridge Studies in Advanced Mathematics}.
\newblock Cambridge University Press, Cambridge, 1997.
\newblock With a foreword by Gian-Carlo Rota, Corrected reprint of the 1986
  original.

\bibitem{Stein2}
I.~Stein.
\newblock Algebras of {E}hresmann semigroups and categories.
\newblock {\em Semigroup Forum}, 95(3):509--526, 2017.

\bibitem{mobius1}
B.~Steinberg.
\newblock M{\"o}bius functions and semigroup representation theory.
\newblock {\em J. Combin. Theory Ser. A}, 113(5):866--881, 2006.

\bibitem{mobius2}
B.~Steinberg.
\newblock M{\"o}bius functions and semigroup representation theory. {II}.
  {C}haracter formulas and multiplicities.
\newblock {\em Adv. Math.}, 217(4):1521--1557, 2008.

\bibitem{repbook}
B.~Steinberg.
\newblock {\em Representation theory of finite monoids}.
\newblock Universitext. Springer, Cham, 2016.

\bibitem{Wenger}
R.~Wenger.
\newblock Some semigroups having quasi-{F}robenius algebras. {II}.
\newblock {\em Canadian J. Math.}, 21:615--624, 1969.

\bibitem{Wilf}
H.~S. Wilf.
\newblock Hadamard determinants, {M}\"{o}bius functions, and the chromatic
  number of a graph.
\newblock {\em Bull. Amer. Math. Soc.}, 74:960--964, 1968.

\bibitem{wood}
J.~A. Wood.
\newblock Duality for modules over finite rings and applications to coding
  theory.
\newblock {\em Amer. J. Math.}, 121(3):555--575, 1999.

\bibitem{woodsemigroup}
J.~A. Wood.
\newblock Factoring the semigroup determinant of a finite commutative chain
  ring.
\newblock In {\em Coding theory, cryptography and related areas ({G}uanajuato,
  1998)}, pages 249--259. Springer, Berlin, 2000.

\end{thebibliography}
\def\malce{\mathbin{\hbox{$\bigcirc$\rlap{\kern-7.75pt\raise0,50pt\hbox{${\tt
  m}$}}}}}\def\cprime{$'$} \def\cprime{$'$} \def\cprime{$'$} \def\cprime{$'$}
  \def\cprime{$'$} \def\cprime{$'$} \def\cprime{$'$} \def\cprime{$'$}
  \def\cprime{$'$} \def\cprime{$'$}

\end{document}